\documentclass[11pt,a4paper]{amsart}
\usepackage{amssymb,amsmath,epsfig,graphics,mathrsfs}

\usepackage{fancyhdr}
\usepackage{hyperref}
\hypersetup{
 colorlinks   = true,
 urlcolor     = blue,
 linkcolor    = blue,
 citecolor   = red ,
 bookmarksopen=true
}

\usepackage{amsmath}
\usepackage{amsfonts}
\usepackage{amssymb}
\usepackage{amsthm}
\usepackage{epsfig,graphics,mathrsfs}
\usepackage{graphicx}

\usepackage[usenames, dvipsnames]{color} 

\usepackage{hyperref}

\def \de {\partial}

\def \phi {\varphi}

\def \RN {\mathbb{R}^N}
\def \R {\mathbb{R}}

\def \K {\mathscr{K}}

\def \G{\Gamma}
\newcommand{\Ba}{\mathscr B_z^{(a)}}

\def \vf{\varphi}

\def \So {\mathscr{S}(\R^{N})}
\newcommand{\As}{(-\mathscr A)^s}
\newcommand{\sA}{\mathscr A}

\newcommand{\Bpa}{B^{p,\alpha}\left(\RN\right)}

\newcommand{\Rn}{\mathbb R^n}

\newcommand{\p}{\partial}

\newcommand{\la}{\lambda}

\numberwithin{equation}{section}

\newcommand{\beq}{\begin{equation}}
\newcommand{\bea}[1]{\begin{array}{#1} }
\newcommand{\eeq}{ \end{equation}}
\newcommand{\ea}{ \end{array}}

\newcommand{\ve}{\varepsilon}

\newcommand{\Rnp}{\mathbb R^{N+1}_+}

\newcommand{\Po}{\mathscr P}

\newcommand{\In}{\mathbf 1_E}

\newcommand{\Ia}{\mathscr I_\alpha}

\newtheorem{theorem}{Theorem}[section]
\newtheorem{lemma}[theorem]{Lemma}
\newtheorem{proposition}[theorem]{Proposition}
\newtheorem{corollary}[theorem]{Corollary}
\newtheorem{remark}[theorem]{Remark}
\newtheorem{definition}[theorem]{Definition}

\numberwithin{equation}{section}

\begin{document}

\title[Functional inequalities etc.]{Functional inequalities for a class of nonlocal hypoelliptic equations of H\"ormander type}


\date{}

\begin{abstract} 
We consider a class of second-order partial differential operators $\sA$ of H\"ormander type, which contain as a prototypical example a well-studied operator introduced by Kolmogorov in the '30s. We analyze some properties of the nonlocal operators driven by the fractional powers of $\sA$, and we introduce some interpolation spaces related to them. We also establish sharp pointwise estimates of Harnack type for the semigroup associated with the extension operator. Moreover, we prove both global and localised versions of Poincar\'e inequalities adapted to the underlying geometry.
\end{abstract}

\author{Nicola Garofalo}

\address{Dipartimento d'Ingegneria Civile e Ambientale (DICEA)\\ Universit\`a di Padova\\ Via Marzolo, 9 - 35131 Padova,  Italy}
\vskip 0.2in
\email{nicola.garofalo@unipd.it}

\thanks{The first author was supported in part by a Progetto SID (Investimento Strategico di Dipartimento) ``Non-local operators in geometry and in free boundary problems, and their connection with the applied sciences", University of Padova, 2017.}

\author{Giulio Tralli}
\address{Dipartimento d'Ingegneria Civile e Ambientale (DICEA)\\ Universit\`a di Padova\\ Via Marzolo, 9 - 35131 Padova,  Italy}
\vskip 0.2in
\email{giulio.tralli@unipd.it}

\maketitle

\tableofcontents

\section{Introduction}\label{intro}

In our recent works \cite{GT}, \cite{GThls} we have developed a fractional calculus, and established nonlocal functional inequalities of Hardy-Littlewood-Sobolev type, for the following class of second-order partial differential equations of evolution type
\begin{equation}\label{K0}
\mathscr K u \overset{def}{=} \mathscr A u  - \de_t u = 0,
\end{equation}
with diffusive part in the form
\begin{equation}\label{A0}
\mathscr A u \overset{def}{=} \operatorname{tr}(Q \nabla^2 u) + <BX,\nabla u>.
\end{equation}
 Here, we have denoted by $X$ the variable in $\R^N$ ($N\geq 2$), whereas $Q$ and $B$ indicate two given $N\times N$ matrices with real constant coefficients. For a $N\times N$ matrix $A$ the notation $\operatorname{tr} A$ indicates the trace of $A$, $A^*$ the transpose of $A$, $\nabla^2 u$ the Hessian matrix of $u$. 
 
The aim of the present note is to complement the above cited works, as well as our work in preparation \cite{GTiso}, and also establish two results of independent interest.
We remark that when $Q = I_N$ and $B = O_N$ in \eqref{A0}, then $\sA = \Delta$ and \eqref{K0} gives that $\K = \Delta - \p_t$ is the standard heat operator in $\R^{N+1}$. Although we will at times refer to this classical non-degenerate case for comparison or illustrative purposes, our primary focus is the genuinely degenerate setting in which $Q=Q^*\geq 0$, and $B \not= O_N$. In such framework, the class \eqref{K0} encompasses various evolution equations of interest in mathematics and physics. 

Perhaps the best known example dates  back to Kolmogorov's 1934 note \cite{Kol} on Brownian motion and the theory of gases, and it is given by
\[
\K_0 u = \Delta_v u + <v,\nabla_x u> - \p_t u = 0,
\]
where now $N = 2n$, $X = (v,x)$, with $v, x\in \Rn$. Other examples of degenerate equations in the form \eqref{K0} of interest in physics were studied in \cite{Ch}. We emphasise that the operator $\K_0$ badly fails to be parabolic since it is missing the diffusive term $\Delta_x u$. Nonetheless, it is hypoelliptic. This remarkable fact was proved by Kolmogorov himself, who found the following explicit fundamental solution 
\begin{align*}
p_0(X,Y,t) & = \frac{c_n}{t^{2n}} \exp\big\{- \frac 1t \big(|v-w|^2 
 + \frac 3t <v-w,y-x-tv>  + \frac{3}{t^2} |x- y +tv|^2\big)\big\},
\end{align*}
where $Y = (w,y)$. Since $p_0(X,Y,t)$ is obviously smooth off the diagonal, the hypoellipticity of $\K_0$ follows. 
For the probabilistic meaning of $p_0(X,Y,t)$ we refer the reader to the insightful note of D. Stroock \cite{Str}, from which we now quote: ``\emph{Kolmogorov's example stood in isolation until 1967, when H\"ormander \cite{Ho} proved a general theorem that put it in context}". 

The result referred to in this quote is the celebrated hypoellipticity theorem. Specialised to the class \eqref{K0} such result states that $\K$ is hypoelliptic if and only if for every $t>0$ the covariance matrix is invertible, i.e.,  
\begin{equation}\label{Kt}
K(t) \overset{def}{=} \frac 1t \int_0^t e^{sB} Q e^{s B^\star} ds > 0.
\end{equation}
It is well-known that \eqref{Kt} is equivalent to H\"ormander's famous finite rank condition for the operators \eqref{K0} and \eqref{A0}, see \cite{Ho}, but also \cite{Web, Il, Ku, LP, L}. 

One notable feature of the class \eqref{K0} is that the fundamental solution of $\K$ (i.e., the transition probability kernel of $\sA$) is explicit, see \cite{Ho}. This fact has been extensively used, for example, for interior and boundary regularity issues in \cite{Sc, GL, LP, Pcat94, Manfr, KLT, AT}. In the recent works \cite{GThls, GTiso} such fundamental solution was expressed in the following suggestive form
\begin{equation}\label{PtKt}
p(X,Y,t) = \frac{c_N}{V(t)} \exp\left( - \frac{m_t(X,Y)^2}{4t}\right),
\end{equation}
where $c_N^{-1} = 4^{\frac N2} \G(\frac N2+1)$. In \eqref{PtKt} for $X, Y\in \RN$, $t>0$ and $r>0$, we have used the notation
\begin{equation}\label{m}
m_t(X,Y) = \left(<K(t)^{-1}(Y-e^{tB} X ),Y-e^{tB} X >\right)^{\frac{1}{2}},
\end{equation}
for the non-symmetric intertwined pseudo-distance, and 
\begin{equation}\label{pb}
B_t(X,r) = \{Y\in \RN\mid m_t(X,Y) < r\},
\end{equation}
for the corresponding time-dependent pseudo-balls. Also, we have indicated with
\begin{equation}\label{VS}
V(t) = \operatorname{Vol}_N(B_t(X,\sqrt t)) = \omega_N  (\det(t K(t)))^{1/2}
\end{equation}
the \emph{volume function}. By the expression \eqref{PtKt} it should be apparent that such function is bound to play an important role in the analysis of \eqref{K0}. The reader should note that the right-hand side of \eqref{VS} is independent of the point $X\in \RN$. This reflect the underlying Lie group structure first noted in \cite{LP}. It was shown in \cite{Ho} that for $f\in \So$, $u(X,t) = \int_{\RN}p(X,Y,t)f(Y) dY$ is the unique solution of the Cauchy problem $\K u = 0$ in $\RN\times (0,\infty)$, $u(X,0) = f(X)$. Moreover, the formula 
\begin{equation}\label{hs}
P_t f(X) = e^{-t\sA} f(X) =\int_{\RN}p(X,Y,t)f(Y) dY,
\end{equation}
defines a strongly-continuous semigroup $\{P_t\}_{t>0}$ in every $L^p(\RN)$, $1\le p < \infty$, with infinitesimal generator $-\sA$. The same is true for $p=\infty$, if we agree (as we will, henceforth) to replace $L^\infty(\RN)$ with the Banach space $L^\infty_0(\RN)$ of those $f\in C(\RN)$ such that $\underset{|X|\to \infty}{\lim}\ |f(X)| = 0$ with the norm $||\cdot||_\infty$. In fact, $P_t$ is contractive on $L^\infty_0(\RN)$, but it is not so, in general, on $L^p(\RN)$, when $p<\infty$. For a summary of the main known properties of the semigroup \eqref{hs}, we refer the reader to see \cite[Section 2]{GT} and the bibliography therein.

Despite its apparent similarity with the classical Euclidean heat kernel, formula \eqref{PtKt} hides a greater complexity. One aspect of this is the dependence of both the intertwined pseudo-distance $m_t(X,Y)$, and the pseudo-balls $B_t(X,\sqrt t)$, on the time variable $t$. Another difficulty is the drastically different geometry of $B_t(X,\sqrt t)$ depending on the eigenvalues of the matrix $B$. In this connection, we recall that in \cite[Section 3]{GThls} we showed the $L^p-L^\infty$ ultracontractivity of the semigroup. Precisely, for any $1\le p\le \infty$ one has for $f\in L^p(\RN)$,
\begin{equation}\label{uc}
|P_t f(X)| \le \frac{c_{N,p}}{V(t)^{1/p}} ||f||_{p},
\end{equation}
for a certain constant $c_{N,p}>0$. The unfamiliar reader should be aware that in the general framework of \eqref{K0}, \eqref{A0} this property, per se, does not necessarily imply a decay rate of the semigroup. For instance, when $Q = I_N$ and $B = - I_N$, then the operator in \eqref{K0} becomes $\K = \Delta u - <X,\nabla u> - \p_t u$, the classical Ornstein-Uhlenbeck operator. In such case one can see that $V(t)\to c_N>0$ as $t\to \infty$. One of the main ideas in \cite{GThls} was to combine \eqref{uc} with the assumption
\begin{equation}\label{trace}
\operatorname{tr} B \ge 0.
\end{equation} 
Not only \eqref{trace} guarantees that $P_t$ be contractive in $L^p(\RN)$ for $1\le p<\infty$, see \eqref{P:ptp}, but such assumption also determines the large time behaviour of the volume function $V(t)$. In fact, it was shown in \cite[Prop. 3.1]{GThls} that \eqref{trace} implies that $V(t)$ blows up at least linearly at infinity. Furthermore, if at least one of the eigenvalues of $B$ has a positive real part, then $V(t)$ blows up exponentially. In all cases, when \eqref{trace} holds we obtain from \eqref{uc} that $P_t f(X)\to 0$ as $t\to \infty$. This information played a key role in our proofs of the nonlocal Sobolev and isoperimetric inequalities in \cite{GThls}, \cite{GTiso}. For instance, in establishing the fundamental identity
\begin{equation}\label{azz}
f = \mathscr I_{\alpha} \circ (-\sA)^{\alpha/2} f = (-\sA)^{\alpha/2} \circ \Ia f,\ \ \ \ \ 0<\alpha<2,
\end{equation}
which shows that $\Ia = (-\sA)^{-\alpha/2}$. In \eqref{azz} we have let $\As$ denote the nonlocal operator in \eqref{As}, whereas we have indicated with
\[
\Ia f(X) = \frac{1}{\G(\alpha/2)} \int_0^\infty t^{\alpha/2 - 1} P_t f(X) dt,
\]
the potential operators of order $\alpha$.
While we refer to \cite{GThls, GTiso} for the relevant results, in view of their applications it is of interest to further analyse the properties of the nonlocal operators in \cite{GT}. 

This leads us to briefly discuss the main results in this paper.
In Section \ref{sec2} we recall the definition of $\As$ following the classical approach by Balakrishnan, and we show some notable properties both of $\As$ and an extension of \eqref{uc}. In Section \ref{sec3} we introduce a class of hypoelliptic Besov spaces related to $\sA$. We study the mapping properties of the fractional powers $\As$ from these Besov spaces into $L^p$-spaces, with special attention to indicator functions. This latter aspect becomes relevant in connection with the nonlocal isoperimetric inequalities in \cite{GTiso}. In Section \ref{sec4} we discuss two new Poincar\'e inequalities adapted to the underlying geometry of the operators in \eqref{A0}. In Proposition \ref{Poincare} we prove a global Poincar\'e inequality with respect to the Gaussian kernel in \eqref{PtKt}. In Corollary \ref{C:henri} we establish a localised one on the intertwined pseudoballs in \eqref{pb}. In this latter result the order of differentiation is suitably weighted by the covariance matrix $K(t)$ in \eqref{Kt}. Finally, in Section \ref{sec5} we provide a sharp Harnack inequality for the semigroup of the extension operator associated with $\K$, see Theorem \ref{shth}. Such estimate is deduced from an inequality of Li-Yau type established in Lemma \ref{LYKaPt}.


\section{Fractional powers of $\mathscr A$ and ultracontractivity}\label{sec2}

In this section we recall some results from Balakrishnan's seminal papers  \cite{B59, B}, with the purpose of connecting them to our work \cite{GT}. We also establish a $L^p-L^q$ ultracontractive estimate which extends that in \cite{GThls}, and a new representation of the fractional powers $\As$ using the Poisson kernel in \cite{GT, GThls}.
 
\subsection{The fractional calculus of Balakrishnan}\label{SS:bala}
\addtocontents{toc}{\protect\setcounter{tocdepth}{1}}
To provide a motivation for the definition of the fractional powers, we recall some of the pioneering ideas of Balakrishnan. In his work \cite{B} he considered a closed linear operator $A$ with domain and range in a Banach space $X$. He assumed that every $\la>0$ belongs to the resolvent set $\rho(A)$, and that with $R(\la,A) = (\la I - A)^{-1}$, one has for $\la >0$,
\begin{equation}\label{b1}
\la ||R(\la,A)|| \le M.
\end{equation}
Balakrishnan himself pointed out that this assumption does not necessarily imply that $A$ be the generator of a semigroup on $X$. Under the hypothesis \eqref{b1} he defined in formula (2.1) in \cite{B} a linear operator $J^\alpha:D(A)\subset X\to X$ by the formula
\begin{equation}\label{b2}
J^\alpha x = \frac{\sin(\pi \alpha)}{\pi} \int_0^\infty \la^{\alpha-1} R(\la,A)(-A)x\ d\la,\ \ \ \ \ \ \ 0<\Re \alpha <1.
\end{equation}
If we write the integral in the right-hand side of \eqref{b2} as
\[
\int_0^1 \la^{\alpha-1} R(\la,A)(-A)x d\la + \int_1^\infty \la^{\alpha-1} R(\la,A)(-A)x d\la,
\]
one immediately recognises that both integrals are convergent (in the sense of Bochner). The former does since $\Re \alpha >0$, the latter converges since one has $||\la^{\alpha-1} R(\la,A)(-A)x||\le C \la^{\Re \alpha -2}$ on $[1,\infty)$ thanks to \eqref{b1}, and thus convergence is guaranteed by the hypothesis $\Re \alpha<1$.
On p.421-22, Balakrishnan observed that: ``If $A$ does generate a semigroup, these coincide with the previous definitions in \cite{B59}". Let us provide a proof of this statement for the benefit of the unfamiliar reader. Suppose, in addition to \eqref{b1}, that $A$ be the infinitesimal generator of a strongly continuous semigroup $\{T(t)\}_{t>0}$ on $X$. Then, formula \cite[(6.10) in Theor. 6.3]{B59} gives for $x\in D(A)$,
\begin{equation}\label{b3}
A^\alpha x = - \frac{s}{\G(1-s)} \int_0^\infty \frac{1}{t^{1+\alpha}} (T(t) x - x) dt,  \ \ \ \ \ \ 0<\Re \alpha <1.
\end{equation} 
On the other hand, by the well-known integral representation of the resolvent, see e.g. \cite[(i) in Theor. 1.10]{EN}, we have $R(\la,A) x = \int_0^\infty e^{-\la t} T(t) x\ dt$. We can thus rewrite \eqref{b2} as follows
\begin{equation}\label{b4}
J^\alpha x = \frac{\sin(\pi \alpha)}{\pi} \int_0^\infty \int_0^\infty \la^{\alpha-1} e^{-\la t} T(t)(-A)x\ dt d\la,\ \ \ \ \ \ \ 0<\Re \alpha <1.
\end{equation}
Since $\int_0^\infty \la^{\alpha-1} e^{-\la t} d\la = t^{-\alpha} \G(\alpha)$, exchanging the order of integration, we obtain from \eqref{b4}
\begin{equation}\label{b5}
J^\alpha x = \frac{\sin(\pi \alpha)\G(\alpha)}{\pi} \int_0^\infty  t^{-\alpha} T(t)(-A)x\ dt = - \frac{\sin(\pi \alpha)\G(\alpha)}{\pi} \int_0^\infty  t^{-\alpha} A T(t) x\ dt,
\end{equation}
where in the second equality we have used the fact that $A T(t) x = T(t) Ax$ for every $x\in D(A)$. Keeping in mind that for every $x\in D(A)$ we have $\frac{d}{dt} T(t) x = A T(t) x$, we now proceed as follows
\begin{align*}
\int_0^\infty  t^{-\alpha} A T(t) x\ dt = \int_0^\infty  t^{-\alpha} \frac{d}{dt}[T(t) x - x] dt
= \alpha \int_0^\infty  t^{-1 -\alpha} [T(t) x - x] dt.
\end{align*}
We note that in the above integration by parts the boundary terms at $0$ and $\infty$ vanish since for any $x\in D(A)$ we have $||T(t) x - x|| = O(t)$ as $t\to 0^+$, and $||t^{-\alpha} [T(t) x - x]|| \le C t^{-\Re\alpha}$ as $t\to \infty$ (we note that by the theorem of Hille-Yosida the assumption \eqref{b1} implies that $||T(t)||\le M$ for every $t>0$). Substituting the latter identity in \eqref{b5} we find
$$J^\alpha x = - \frac{\alpha \sin(\pi \alpha)\G(\alpha)}{\pi} \int_0^\infty  t^{-1 -\alpha} [T(t) x - x] dt.$$
Keeping the identity $\G(\alpha) \G(1-\alpha) = \frac{\pi}{\sin \pi \alpha}$ in mind (see e.g. \cite[3.123 on p.105]{T}), we finally reach the conclusion that $J^\alpha x = A^\alpha x$, for every $x\in D(A)$, which proves Balakrishnan's comment. Summarizing, we have shown that when $A$ satisfies \eqref{b1} and it is the infinitesimal generator of a strongly continuous semigroup $\{T(t)\}_{t>0}$ on $X$, then for $0<\Re \alpha <1$ one has
\begin{equation}\label{b7}
A^\alpha x = - \frac{s}{\G(1-s)} \int_0^\infty \frac{1}{t^{1+\alpha}} (T(t) x - x) dt = \frac{\sin(\pi \alpha)}{\pi} \int_0^\infty \la^{\alpha-1} R(\la,A)(-A)x\ d\la.
\end{equation}  

As an illustration of \eqref{b7} (and, in fact, this example was the main motivation behind Balakrishnan's formula \eqref{b3}), consider the standard heat semigroup $e^{-t\Delta}$ in $\RN$, with generator $A = -\Delta$. If we denote by $p(X,Y,t) = (4\pi t)^{-\frac N2} \exp(-\frac{|X-Y|^2}{4t})$ the heat kernel, then for any $f\in \So$ one has $e^{-t\Delta} f(X) = \int_{\RN} p(X,Y,t) f(Y) dY$. If we insert this information in \eqref{b3}, and we exchange the order of integration, we find
\begin{align*}
(-\Delta)^s f(X) & = - \frac{s}{\G(1-s)} \int_0^\infty t^{-(1+s)} [e^{-t\Delta} f(X) - f(X)] dt
\\
& = - \frac{s}{\G(1-s)} \int_{\RN} [f(Y) - f(X)] \int_0^\infty t^{-(1+s)} p(X,Y,t) dt dY
\\
& = - \frac{s 2^{2s} \G(\frac N2 + s)}{\pi^{\frac N2} \G(1-s)} \operatorname{PV} \int_{\RN} \frac{f(Y) - f(X)}{|Y-X|^{N+2s}} dY, 
\end{align*}
where in the last integral we have used the well-known identity
\begin{equation}\label{flheat}
- \frac{s}{\G(1-s)} \int_0^\infty t^{-(1+s)} p(X,Y,t) dt = - \frac{s 2^{2s} \G(\frac N2 + s)}{\pi^{\frac N2} \G(1-s)} |X-Y|^{-(N+2s)}.
\end{equation}
We conclude that Balakrishnan's formula \eqref{b3} coincides with M. Riesz' definition in \cite{R} of the fractional Laplacian of a function $f\in \So$ 
$$(-\Delta)^s f(X) = - \frac{s 2^{2s} \G(\frac N2 + s)}{\pi^{\frac N2} \G(1-s)} \operatorname{PV} \int_{\RN} \frac{f(Y) - f(X)}{|Y-X|^{N+2s}} dY,$$
see also \cite{La}, and the survey articles \cite{BV}, \cite{Kwa}, \cite{Gft}, \cite{AV}. 

\subsection{The fractional powers $\As$}
After these preliminaries we return to the semigroup \eqref{hs}. It was shown in \cite{GT} that, under the assumption \eqref{trace} such semigroup possesses all the properties which are needed to the implementation of Balakrishnan's fractional calculus. We recall here the definition of the nonlocal operators $\As$ given in \cite[Definition 3.1]{GT}. 

\begin{definition}\label{D:flheat}
Let $0<s<1$. For any $f\in \mathscr S(\R^N)$ we define for $X\in \RN$,
\begin{align}\label{As}
(-\mathscr A)^s f(X) & =  - \frac{s}{\G(1-s)} \int_0^\infty t^{-(1+s)} \left[P_t f(X) - f(X)\right] dt.
\end{align}
\end{definition}
The previous definition makes a pointwise sense for any drift matrix $B$, and it makes sense also in $L^p$ for any $p\in[1,+\infty]$ if $B$ satisfies \eqref{trace}.
We next collect some known basic properties of the nonlocal operators \eqref{As}, see \cite[Lemmas 2.3, 2.4 and 2.5]{B}. In the next result, when we write $D(\sA)$ we mean the domain of the infinitesimal generator of the semigroup $P_t$ over the Banach space $L^\infty_0(\RN)$.
 
\begin{proposition}\label{P:bala2}
The following properties hold:
\begin{itemize}
\item[(i)] for any $f\in D(\sA)$ such that $\sA f\in \overline{D(\sA)}$, one has $\underset{s\to 1^-}{\lim} \As f = - \sA f$;
\item[(ii)] for any $f\in D(\sA)$ such that $\la R(\la,\sA)\to 0$ as $\la\to 0^+$, one has $\underset{s\to 0^+}{\lim} \As f = - f$;
\item[(iii)] let $s, s'\in (0,1)$ and suppose that $s+s'\in (0,1]$. Then, for any $f\in D(\sA^2)$ one has
\[
(-\mathscr A)^{s+s'} f = (-\mathscr A)^s \circ (\mathscr A)^{s'} f.
\]
\end{itemize}
\end{proposition}

\subsection{Ultracontractivity}\label{SS:ultra}

The role of the number $\operatorname{tr} B$ in the spectral properties of the semigroup $\{P_t\}_{t>0}$ can be seen in the following estimate valid for any $1\le p \le \infty$, and any $t>0$ (see, e.g, \cite[Lemma 2.4, (iv)]{GT}),
\begin{equation}\label{P:ptp}
||P_t ||_{L^p(\R^N)\to L^p(\R^N)} \le e^{- t \frac{\operatorname{tr} B}p}.
\end{equation}
Moreover, we proved in \cite[Proposition 3.5]{GThls} that for any $1\leq p< \infty$ the following $L^p\to L^\infty$ ultracontractivity holds
\begin{equation}\label{L:Koneinfty}
||P_t ||_{L^p(\R^N)\to L^\infty(\R^N)}\le \frac{c_{N,p}}{V(t)^{\frac 1p}},
\end{equation}
for a certain constant $c_{N,p}>0$. In the next proposition we generalise formulas \eqref{P:ptp}-\eqref{L:Koneinfty} by establishing the following $L^p\to L^q$ ultracontractivity of the semigroup $\{P_t\}_{t>0}$.
\begin{proposition}\label{P:ptq}
For every $1\leq p< \infty$ and $q\geq p$,  we have $P_t: L^p(\R^N)\to L^q(\R^N)$ for any $t>0$, with
\begin{equation}\label{tuttenorme}
||P_t ||_{L^p(\R^N)\to L^q(\R^N)} \le \frac{C(N,p,q)}{V(t)^{\frac{1}{p}-\frac{1}{q}}}e^{- t \frac{\operatorname{tr} B}{q}},
\end{equation}
for some constant $C(N,p,q)>0$.
\end{proposition}
\begin{proof}
Let $r\ge 1$ be arbitrarily fixed at this moment. If $f\in L^1(\RN)$, then Minkowski's integral inequality gives
\begin{align*}
& \left(\int_{\RN} |P_t f(X)|^r dX\right)^{1/r} = \left(\int_{\RN} \left|\int_{\RN} p(X,Y,t) f(Y) dY\right|^r dX\right)^{1/r}
\\
& \le \int_{\RN} |f(Y)| \left(\int_{\RN} p(X,Y,t)^r dX\right)^{1/r} dY = ||f||_1 \left(\int_{\RN} p(X,Y,t)^r dX\right)^{1/r} dY.
\end{align*} 
We can compute explicitly the $L^r$-norm of $p(\cdot,Y,t)$. In fact, using the expression \eqref{PtKt} and \cite[Lemma 2.1, (2.3)]{GThls}, we find
$$\left(\int_{\RN} p(X,Y,t)^r dX\right)^{1/r}=\frac{c_{N,r}}{V(t)^{1-\frac{1}{r}}}e^{- t \frac{\operatorname{tr} B}{r}}.$$
This shows that
\[
P_t : L^1(\RN)\ \longrightarrow\ L^r(\RN),
\]
with
\[
||P_t||_{L^1(\R^N)\to L^r(\R^N)} \le \frac{\tilde{c}_{N,r}}{V(t)^{1-\frac{1}{r}}}e^{- t \frac{\operatorname{tr} B}{r}}
\]
for some $\tilde{c}_{N,r}>0$. On the other hand, we know from \eqref{L:Koneinfty} that
\[
P_t : L^{r'}(\RN)\ \longrightarrow\ L^\infty(\RN)\qquad\mbox{with}\qquad||P_t||_{L^{r'}\to L^\infty} \le \frac{c_{N,r'}}{V(t)^{\frac{1}{r'}}}=\frac{c_{N,r'}}{V(t)^{1-\frac{1}{r}}}.
\]
Let now $1\le p \le q$ be fixed, and choose $r\ge 1$ such that $r'\ge p$. Then, there exists $\la\in [0,1]$ such that
\[
\frac 1p = \frac{1-\la}{1} + \frac{\la}{r'} = 1 - \frac{\la}{r}.
\] 
In other words, we are taking $\frac{\la}r = 1-\frac 1p = \frac{1}{p'}$. By the Riesz-Thorin interpolation theorem 
\[
P_t : L^p(\RN) \ \longrightarrow\ L^q(\RN),
\]
with 
\[
\frac 1q = \frac{1-\la}r + \frac{\la}{\infty} = \frac 1r - \frac{\la}r = \frac 1p + \frac 1r - 1.
\] 
Moreover, since $\frac{1-\la}{r} = \frac 1q$ and $1 - \frac{\la}{r}=\frac 1p$, we have
\[
||P_t||_{L^p \to L^q} \le \left(\frac{\tilde{c}_{N,r}}{V(t)^{1-\frac{1}{r}}}e^{- t \frac{\operatorname{tr} B}{r}}\right)^{1-\la} \left(\frac{c_{N,r'}}{V(t)^{1-\frac{1}{r}}}\right)^{\la}=\frac{C(N,p,q)}{V(t)^{\frac{1}{p}-\frac{1}{q}}}e^{- t \frac{\operatorname{tr} B}{q}},
\]
where $C(N,p,q)=\tilde{c}^{1-\la}_{N,r}c^\la_{N,r'}$. We have thus reached the desired conclusion \eqref{tuttenorme}.
\end{proof}

We close this section by providing the following alternative expression of the nonlocal operator $\As$ based on the Poisson semigroup $\Po_z = e^{z \sqrt{-\sA}}$, $z>0$. We recall from \cite{GT} that the Poisson semigroup is defined by
\begin{equation}\label{altrep1}
\Po_z f(X) = \frac{1}{\sqrt{4\pi}} \int_0^\infty \frac{z}{t^{3/2}} e^{-\frac{z^2}{4t}} P_t f(X) dt.
\end{equation}

\begin{proposition}\label{P:Aspoisson}
Let $0<s<1/2$. For $f\in \So$ we have 
\begin{equation}\label{b8}
\As f(X)  = - \frac{2s}{\G(1-2s)} \int_0^\infty \frac{1}{z^{1+2s}} [\Po_z f(X) - f(X)] dz.
\end{equation}
\end{proposition}

\begin{proof}
To verify \eqref{b8},
we have from \eqref{altrep1}
\begin{align*}
\int_0^\infty \frac{1}{z^{1+2s}} [\Po_z f(X) - f(X)] dz = 
\int_0^\infty \frac{1}{z^{1+2s}} \frac{1}{\sqrt{4\pi}} \int_0^\infty \frac{z}{t^{3/2}} e^{-\frac{z^2}{4t}} [P_t f(X)-f(X)]\ dt dz,
\end{align*}
where we have used the fact that 
\[
\frac{1}{\sqrt{4\pi}} \int_0^\infty \frac{z}{t^{3/2}} e^{-\frac{z^2}{4t}}\ dt = \frac{1}{\sqrt{\pi}} \int_0^\infty \frac{z}{\sqrt{4t}} e^{-\frac{z^2}{4t}}\ \frac{dt}t = 1.
\]
Suppose now that $0<s<1/2$. Exchanging the order of integration in the above integral, we find
\begin{align*}
& \int_0^\infty \frac{1}{z^{1+2s}} [\Po_z f(X) - f(X)] dz = \frac{1}{\sqrt{4\pi}} \int_0^\infty \frac{1}{t^{3/2}} [P_t f(X)-f(X)] \int_0^\infty z^{1-2s} e^{-\frac{z^2}{4t}} \frac{dz}z dt
\\
& = \frac{1}{4\sqrt{\pi}} \int_0^\infty \frac{1}{t^{3/2}} [P_t f(X)-f(X)] \int_0^\infty (4tu)^{1/2-s} e^{-u} \frac{du}u dt
\\
& = \frac{1}{2^{1+2s}\sqrt{\pi}} \int_0^\infty \frac{1}{t^{1+s}} [P_t f(X)-f(X)] \int_0^\infty u^{1/2-s} e^{-u} \frac{du}u dt
\\
& = \frac{\G(1/2 - s)}{2^{1+2s}\sqrt{\pi}} \int_0^\infty \frac{1}{t^{1+s}} [P_t f(X)-f(X)] dt = - \frac{\G(1/2 - s)\G(1-s)}{2s\ 2^{2s}\sqrt{\pi}} \As f(X).
\end{align*}
Using the formula
\[
2^{2z-1} \G(z) \G(z+\frac 12) = \sqrt \pi \G(2z),
\]
with $z = 1-s$, we obtain
\[
2^{-2s} (1 - 2s) \G(1-s) \G(1/2 -s) = 2^{1-2s} \G(1-s) \G(1-s+1/2) = \sqrt \pi \G(1-2s+1) = \sqrt \pi (1-2s)\G(1-2s),
\]
which gives
\[
2^{-2s} \G(1-s) \G(1/2 -s) = \sqrt \pi \G(1-2s).
\]
Substituting in the above equation, we conclude that \eqref{b8} is valid.
\end{proof}


\section{Interpolation spaces and fractional powers of $\mathscr A$}\label{sec3}

In this section we introduce a class of hypoelliptic Besov spaces which are tailored on the operator $\mathscr A$. By this we mean that, as we show in Proposition \ref{P:besov} below, for every $p\ge 1$ the fractional powers $(-\mathscr A)^s$ continuously map the Besov space $B^{p,\alpha}(\RN)$ into $L^p(\RN)$, provided that $\alpha >2s$ and that \eqref{trace} hold. We then turn our attention to the Besov seminorm of indicator functions, which is relevant for the theory of nonlocal perimeters developed in \cite{GTiso}.

To motivate our definition we recall the classical Sobolev-Besov spaces, see e.g. \cite{JW}, the monograph \cite{DGN} and the references therein. When $\alpha>0$ and $1\le p<\infty$ the space $B^p_\alpha(\RN) = B^{p,p}_\alpha(\RN)$ is the collection of all functions $f\in L^p(\RN)$ such that the seminorm
\begin{equation}\label{normabesovnorma}
\mathscr N^\Delta_{p,\alpha}(f) = \left(\int_{\RN} \int_{\RN} \frac{|f(X) - f(Y)|^p}{|X-Y|^{N+\alpha p}} dX dY\right)^{1/p} < \infty.
\end{equation}
Seminorms of this sort were considered by Slobedetzky \cite{Ni}, Aronszajn \cite{A} and Gagliardo \cite{Ga}.
Using \eqref{flheat} with $2s = \alpha p$, with the aid of the classical heat kernel $p(X,Y,t)$, we can express the condition $\mathscr N^\Delta_{p,\alpha}(f)<\infty$ in the alternative form
\begin{equation}\label{alt1}
\left(\int_0^\infty \frac{1}{t^{\frac{\alpha p}2}} \int_{\RN} \int_{\RN} p(X,Y,t) |f(X) - f(Y)|^p dX dY \frac{dt}t\right)^{\frac 1p} < \infty.
\end{equation}
Keeping in mind that $e^{-t\Delta} f(X) = \int_{\RN} p(X,Y,t) f(Y) dY$, we see that \eqref{alt1} can be equivalently formulated as follows:
 \begin{equation}\label{alt2}
\left(\int_0^\infty \frac{1}{t^{\frac{\alpha p}2}} \int_{\RN} e^{-t\Delta}\left(|f - f(Y)|^p\right)(Y) dY \frac{dt}t\right)^{\frac 1p} < \infty.
\end{equation}
Concerning \eqref{alt2}, we recall that the first mathematician to introduce a characterisation of the Besov spaces via the Poisson or the Gauss-Weierstrass kernels was M. Taibleson in his seminal work \cite{T1}, see also \cite{T2}. 
 
Having observed \eqref{alt1} and \eqref{alt2} for the classical case, we now return to the setting of \eqref{A0} and use the semigroup \eqref{hs} to introduce the relevant definition.

\begin{definition}\label{D:besov}
For $p\geq 1$ and $\alpha\geq 0$, we define the \emph{Besov space} $\Bpa$ as the collection of those functions $f\in L^p(\RN)$, such that the seminorm
\begin{equation}\label{defBnorm}
\mathscr N_{p,\alpha}(f) = \left(\int_0^\infty  \frac{1}{t^{\frac{\alpha p}2}} \int_{\RN} P_t\left(|f - f(Y)|^p\right)(Y) dY \frac{dt}t\right)^{\frac 1p} < \infty.
\end{equation}
We endow the space $\Bpa$ with the following norm
$$||f||_{\Bpa} \overset{def}{=} ||f||_{L^p(\RN)} + \mathscr N_{p,\alpha}(f).$$
\end{definition}

\begin{remark}\label{R:01}
Suppose that \eqref{trace} hold. Then, for any $p\ge 1$ and $\alpha>0$ the condition $\mathscr N_{p,\alpha}(f)<\infty$ is equivalent to 
\[
\tilde{\mathscr N}_{p,\alpha}(f) = \left(\int_0^1  \frac{1}{t^{\frac{\alpha p}2}} \int_{\RN} P_t\left(|f - f(Y)|^p\right)(Y) dY \frac{dt}t\right)^{\frac 1p} < \infty.
\]
\end{remark}
To see the remark, assume that $\tilde{\mathscr N}_{p,\alpha}(f)<\infty$. 
Then, from \eqref{P:ptp} and $P_t1=1$, we have
\begin{align*}
& \left(\int_1^\infty  \frac{1}{t^{\frac{\alpha p}2}} \int_{\RN} P_t\left(|f - f(Y)|^p\right)(Y) dY \frac{dt}t\right)^{\frac 1p}
\\
& \le 2^{1-\frac 1p} \left(\int_1^\infty  \frac{1}{t^{\frac{\alpha p}2}} \int_{\RN} \left(P_t\left(|f|^p\right)(Y) + |f(Y)|^p P_t 1(Y)\right)dY \frac{dt}t\right)^{\frac 1p}
\\
& \le 2^{1-\frac 1p} \left(\int_1^\infty  \frac{1}{t^{\frac{\alpha p}2}} \left(e^{-t \operatorname{tr} B} \int_{\RN} |f(Y)|^p dY+ \int_{\RN} |f(Y)|^p dY\right) \frac{dt}t\right)^{\frac 1p}
\\
& \le 2 ||f||_{L^p(\RN)} \left(\int_1^\infty  \frac{dt}{t^{1+\frac{\alpha p}2}} \right)^{\frac 1p},
\end{align*}
where we have used \eqref{trace}. This proves that there exists a constant $C(p,\alpha)>0$ such that
\begin{equation}\label{alt}
\mathscr N_{p,\alpha}(f) \le C \left(\tilde{\mathscr N}_{p,\alpha}(f) + ||f||_p\right) < \infty,
\end{equation}
which implies the desired conclusion.

The following result establishes the mapping properties of the fractional powers of $\sA$ from a Besov space $B^{p,\alpha}(\RN)$ into $L^p(\RN)$ (in this respect, see also \cite[Lemma 4.3]{GThls}).

\begin{proposition}\label{P:besov}
Assume \eqref{trace}, and let $0<s<1$. For $p>1$ and $\alpha > 2s$, we have
\begin{equation}\label{mappingpowers}
\As: B^{p,\alpha}\left(\RN\right) \to L^p\left(\RN\right).
\end{equation}
When $p=1$, we have for $\alpha \geq 2s$,
\begin{equation}\label{mappingpowerspugu1}
\As: B^{1,\alpha}\left(\RN\right) \to L^1\left(\RN\right).
\end{equation}
\end{proposition}

\begin{proof}
Let $p\geq 1$ and $f\in B^{p,\alpha}\left(\RN\right)$. We notice that
$$X\longmapsto \frac{-s}{\G(1-s)} \int_0^\infty t^{-1-s}\left(P_tf(X) - f(X)\right)\,{\rm d}t=: \As f(X)$$
is a measurable function. Furthermore,
\begin{align*}
\left\|\As f\right\|_p &\leq \frac{s}{\G(1-s)} \int_0^\infty t^{-1-s}\left\|P_tf - f\right\|_p\,{\rm d}t \\
& \leq \int_0^1 t^{-1-s}\left\|P_tf - f\right\|_p\,{\rm d}t + \int_1^\infty t^{-1-s}\left\|P_tf - f\right\|_p\,{\rm d}t.
\end{align*}
If \eqref{trace} holds, we have
\[
\int_1^\infty t^{-1-s}\left\|P_tf - f\right\|_p\,{\rm d}t \le 2 ||f||_p \int_1^\infty t^{-1-s} dt = \frac 2s ||f||_p.
\]
Therefore, in view of Remark \ref{R:01}, in order to establish \eqref{mappingpowers}, \eqref{mappingpowerspugu1}, it suffices to bound the integral $\int_0^1 t^{-1-s}\left\|P_tf - f\right\|_p\,{\rm d}t$ in terms of $\tilde{\mathscr N}_{\alpha,p}(f)$. With this objective in mind, let us observe that H\"older inequality and the fact that $P_t1=1$, give for any $p\ge 1$,
\begin{equation}\label{hol}
\left\|P_tf - f\right\|_p\leq \left(\int_{\RN} P_t\left(|f - f(Y)|^p\right)(Y) dY\right)^{\frac{1}{p}}.
\end{equation}
Now, if $p=1$ and $\alpha\geq 2s$, we find
\begin{align*}
& \int_0^1 t^{-1-s}\left\|P_tf - f\right\|_1 dt \leq \int_0^1 \frac{t^{\frac{\alpha }{2}-s}}{t^{\frac{\alpha }{2}}} \int_{\RN} P_t\left(|f - f(Y)|\right)(Y) dY\frac{dt}{t}
\\
& \le \int_0^1 \frac{1}{t^{\frac{\alpha }{2}}} \int_{\RN} P_t\left(|f - f(Y)|\right)(Y) dY\frac{dt}{t}  = \tilde{\mathscr N}_{\alpha,1}(f).
\end{align*}
This proves \eqref{mappingpowerspugu1}. If instead $p>1$, for every $\alpha>2s$ we obtain from H\"older inequality and \eqref{hol},
\begin{align*}
& \int_0^1 t^{-1-s}\left\|P_tf - f\right\|_p dt \leq \left( \int_0^1 t^{-1-\left(s-\frac{\alpha}{2}\right)p'}dt\right)^{\frac{1}{p'}}\left(\int_0^1 t^{-1-\frac{\alpha p}{2}} \int_{\RN} P_t\left(|f - f(Y)|^p\right)(Y) dY dt\right)^{\frac{1}{p}}
\\
& = C(p,\alpha, s) \tilde{\mathscr N}_{\alpha,p}(f).
\end{align*}
In view of \eqref{alt}, this proves \eqref{mappingpowers}.

\end{proof}

\subsection{Besov spaces and indicator functions}\label{SS:bi}
We now take a closer look at the cases $p=1$ and $p =2$ of Definition \ref{D:besov}, with particular attention to indicator functions $f=\mathbf 1_E$ of measurable sets $E\subset \RN$. For any $1\le p<\infty$ and $0<s<1$ we denote by 
$$D_{p,s} = \{f\in L^p(\RN)\,:\, \As f \in L^p(\RN)\},$$
the domain of $\As$ in $L^p(\RN)$. Throughout this section, we use the notation $C(s) = \frac{s}{\G(1-s)}>0$.

\begin{proposition}\label{L:L1As}
Let $s\in (0,1)$ and $E\subset \RN$ be a measurable set such that $\mathbf 1_E\in D_{1,s}$. Then,
\begin{align}\label{besovind}
||(-\sA)^s \In||_{L^1(\RN)} & = C(s) \int_0^\infty{\frac{1}{t^{1+s}}\int_{\RN}{P_t(|\In-\In(X)|)(X)dX}dt}
\\
& = C(s) \int_0^\infty{\frac{1}{t^{1+s}}\int_{\RN}{P_t(|\In-\In(X)|^2)(X)dX}dt}.
\notag
\end{align}
\end{proposition}

\begin{proof}
Let us first observe that, since $P_t1=1$, one has
\begin{align}\label{altL1}
||P_t\In - \In||_{L^1(\RN)} &= \int_E\int_{\RN\smallsetminus E}p(X,Y,t)dYdX  + \int_{\RN\smallsetminus E}\int_{E}p(X,Y,t)dYdX 
\\
& =  \int_{\RN}{P_t(|\In-\In(X)|)(X)dX}.
\notag
\end{align}
From \eqref{As} and \eqref{altL1} we now have
\begin{align*}
&||(-\sA)^s \In||_{L^1(\RN)}= C(s) \int_{\RN}\left|\int_0^\infty\frac{1}{t^{1+s}}\left(P_t\In(X)-\In(X)\right)dt\right|dX\\
&= C(s) \left(\int_{E}\int_0^\infty\frac{1}{t^{1+s}}(1-P_t\In(X))dtdX + \int_{\RN\setminus E}\int_0^\infty\frac{1}{t^{1+s}}P_t\In(X)dtdX\right)
\\
& = C(s) \int_0^\infty \frac{1}{t^{1+s}} ||P_t\In - \In||_{L^1(\RN)} dt
\\
&= C(s) \int_0^\infty{\frac{1}{t^{1+s}}\int_{\RN}{P_t(|\In-\In(X)|)(X)dX}dt}
\\
&= C(s) \int_0^\infty{\frac{1}{t^{1+s}}\int_{\RN}{P_t(|\In-\In(X)|^2)(X)dX}dt}.
\end{align*}
\end{proof}

\begin{remark}\label{R:alteAs}
Suppose that $\In \in D_{1,s}$. We note explicitly the following alternative expression that follows from \eqref{besovind} and  \eqref{altL1}, 
\[
||(-\sA)^s \In||_{L^1(\RN)} =  C(s) \int_0^\infty \frac{1}{t^{1+s}} ||P_t\In - \In||_{L^1(\RN)} dt.
\]
\end{remark} 

The next result provides a basic characterisation, in terms of the Besov spaces $\Bpa$, for membership of an indicator function in the domain of $\As$ in $L^1(\RN)$. We stress that the next proposition is not true for general functions.

\begin{corollary}\label{P:besovchar}
Let $E\subset \RN$ be measurable. The following are equivalent:
\begin{itemize}
\item[(i)] $\mathbf 1_E \in D_{1,s}$;
\item[(ii)] $\mathbf 1_E  \in B^{2,s}\left(\RN\right)$;
\item[(iii)] $\mathbf 1_E  \in B^{1,2s}\left(\RN\right)$.
\end{itemize}
Furthermore, when either one of these equivalent statements hold, we have
\begin{equation}\label{besovind2}
||(-\sA)^s \In||_{L^1(\RN)}= C(s) \mathscr N_{2,s}(\mathbf 1_E)^2 = C(s) \mathscr N_{1,2s}(\mathbf 1_E).  
\end{equation}
\end{corollary}

\begin{proof}
Proposition \ref{L:L1As} and \eqref{besovind} immediately give (i)\ $\Longrightarrow$\ (ii)\ $\Longleftrightarrow$\ (iii). To complete the proof suppose (iii) hold. According to \eqref{defBnorm} this means that $\In\in L^1(\RN)$ and 
\begin{align*}
\mathscr N_{1,2s}(\In)  = \int_0^\infty  \frac{1}{t^{1+ s}} \int_{\RN} P_t\left(|\In - \In(X)|\right)(X) dX dt  = \int_0^\infty  \frac{1}{t^{1+ s}} ||P_t\In - \In||_{L^1(\RN)} dt 
< \infty,
\end{align*}
where in the second equality we have used \eqref{altL1}. By Remark \ref{R:alteAs} we infer
\[
||\As \In||_1 = C(s) \mathscr N_{1,2s}(\In) <\infty.
\]
This implies that $\As \In\in L^1(\RN)$, hence (i) holds. 

\end{proof}

At this point it is worth recalling some well-known facts concerning the case $p = 2$ for the classical Besov spaces. From \eqref{normabesovnorma} and Plancherel theorem we have
\begin{align*}
\mathscr N^\Delta_{2,\alpha}(f)^2 & = \int_{\RN} \int_{\RN} \frac{|f(X) - f(Y)|^2}{|X-Y|^{N+2\alpha}} dX dY = \int_{\RN} \int_{\RN} \frac{|f(Y+h) - f(Y)|^2}{|h|^{N+2\alpha}} dY dh
\\
& = \int_{\RN} \int_{\RN} |e^{2\pi i<h,\xi>} - 1|^2 |\hat f(\xi)|^2 d\xi \frac{dh}{|h|^{N+2\alpha}}
\\
& = 2 \int_{\RN} |\hat f(\xi)|^2 \int_{\RN} \frac{1 - \cos(2\pi<h,\xi>)}{|h|^{N+2\alpha}} dh d\xi.
\end{align*}
Now, a simple computation gives
\begin{align*}
& \int_{\RN} \frac{1 - \cos(2\pi<h,\xi>)}{|h|^{N+2\alpha}} dh = (2 \pi |\xi|)^{2\alpha} \int_{\RN} \frac{1 - \cos(h_N)}{|h|^{N+2\alpha}} dh
\\
& = \frac{\pi^{\frac N2} \G(1-s)}{s 2^{2s} \G\left(\frac{N+ 2s}{2}\right)} (2 \pi |\xi|)^{2\alpha},
\end{align*}
where in the last equality we have used the well-known identity
\[
\int_{\RN} \frac{1 - \cos(h_N)}{|h|^{N+2\alpha}} dh =  \frac{\pi^{\frac N2} \G(1-\alpha)}{\alpha 2^{2\alpha} \G\left(\frac{N+ 2\alpha}{2}\right)},
\]
see e.g. \cite[Propositions 5.1 and 5.6]{Gft}. Substituting in the above, using the identity $\widehat{(-\Delta)^{\alpha/2} f}(\xi) = (2 \pi |\xi|)^{\alpha} \hat f(\xi)$, and Plancherel theorem again, we conclude
\begin{equation}\label{cute}
\mathscr N^\Delta_{2,\alpha}(f)^2  =  \frac{2^{1-2\alpha} \pi^{\frac N2} \G(1-\alpha)}{\alpha \G\left(\frac{N+ 2\alpha}{2}\right)} ||(-\Delta)^{\alpha/2} f||_{L^2(\RN)}. 
\end{equation}
Finally, from \eqref{cute} we find
\begin{equation}\label{verycute}
\mathscr N^\Delta_{2,\alpha}(\In)^2  =  \frac{2^{1-2\alpha} \pi^{\frac N2} \G(1-\alpha)}{\alpha \G\left(\frac{N+ 2\alpha}{2}\right)} ||(-\Delta)^{\alpha/2} \In||_{L^2(\RN)}. 
\end{equation}
We mention that the left-hand side of \eqref{verycute} is what in their seminal work \cite{CRS} Caffarelli, Roquejoffre and Savin call the nonlocal perimeter of a measurable set $E\subset \RN$. Precisely, for every $0<s<1/2$ the $s$-\emph{perimeter} of $E$ is 
\begin{equation}\label{sper}
P_s(E) \overset{def}{=}  \mathscr N^\Delta_{2,s}(\In)^2 = \mathscr N^\Delta_{1,2s}(\In),
\end{equation} 
where in the second equality we have used \eqref{besovind2}.

\begin{remark}\label{R:besov}
To understand the limitation $0<s<1/2$, we stress that if $E$ is a non-empty open set $E$, such that $|E|<\infty$, then in view of \eqref{verycute}, \eqref{sper}, we have
\[
P_s(E) < \infty \Longleftrightarrow ||(-\Delta)^{s/2} \In||_{L^2(\RN)} < \infty \Longleftrightarrow \int_{\RN} |\xi|^{2s} |\hat{\mathbf 1}_E(\xi)|^2 d\xi < \infty.
\]
To now see why the condition $P_s(E) < \infty$ imposes the restriction $0<s<1/2$, consider e.g. the unit ball $B = \{x\in \RN\mid |X|<1\}$. Using Bochner's formula 
$$
\hat u(\xi)=2\pi|\xi|^{-\frac{N}2 +1}\int^\infty_0 r^{\frac{N}2} f(r)J_{\frac{N}2-1}
(2\pi|\xi|r) dr
$$
for the Fourier transform of a spherically symmetric function $u(X) = f(|X|)$, see  \cite[Theorem 40 on p. 69]{BC}, in combination with the identity
\[
\int_0^1 x^{\nu+1} J_\nu(a x) dx = a^{-1} J_{\nu +1}(a),\ \ \ \ \ \ \ \ \ \Re \nu > -1,
\]
see \cite[6.561, 5., p.683]{GR}, we find $\hat{\mathbf 1}_B(\xi) = |\xi|^{-\frac N2} J_{\frac N2}(2\pi 
|\xi|)$,
where $J_\nu(z)$ is the Bessel function of the first kind and order $\nu$. Since the asymptotic behaviour of $J_\nu$ is given by
\[
J_\nu(z)\cong\frac{2^{-\nu}}{\Gamma(\nu+1)}z^\nu,\quad\text{as }z\to 0,\ \ J_\nu(z) =\sqrt{\frac2{\pi
z}}\cos\left(z-\frac{\pi\nu}2-\frac\pi4\right)+
O(z^{-\frac32})\quad\text{as }z\to+\infty,
\]
we see that $|\xi|^{s} \hat{\mathbf 1}_B(\xi) \in L^2(\RN)$ if and only if $0<s<1/2$. More in general, for any non-empty open set $E\subset \RN$, one has $(-\Delta)^{s/2} \In \not\in L^2(\RN)$ for $s = 1/2$, see \cite[Lemma 3.2]{Sickel}. 
\end{remark}


 \section{A Poincar\'e inequality in Gaussian space}\label{sec4}

In 1968 D. Aronson established the following off-diagonal Gaussian lower bound for the fundamental solution $p(x,y,t)$ of a divergence form uniformly parabolic equation with bounded measurable coefficients 
\[
C t^{-\frac n2} \exp\left(- \alpha \frac{|x-y|^2}{t}\right) \le p(x,y,t),
\]
see \cite{A68}. His proof was based on the Harnack inequality for parabolic equations which had been recently established by J. Moser in \cite{Mo1}, \cite{Mo2}. In their work \cite{FS}  Fabes and Stroock's gave a beautiful rendition of the groundbreaking 1958 ideas of Nash in \cite{nash}. In their approach the above Gaussian lower bound was obtained independently from the Harnack inequality. One crucial tool was the following Poincar\'e inequality with respect to the Gaussian measure that had already played a key role in Nash' seminal work: \emph{for every $f\in \mathcal S(\Rn)$ one has}
\begin{equation}\label{pg}
\int_{\Rn} |f - a_f|^2 d\mu \le 2 \int_{\Rn} |\p f|^2 d\mu.
\end{equation}
In \eqref{pg} we have denoted by $d\mu(x) = (4\pi)^{-\frac n2} e^{-\frac{|x|^2}{4}} dx$ the normalised Gaussian measure in $\Rn$, and we have let $a_f = \int_{\Rn} f d\mu$.

The objective of this section is to prove the following generalisation of the inequality \eqref{pg}, when Gaussian measure is replaced by the transition density kernel $p(X,Y,t)$ in \eqref{hs}.

\begin{proposition}[Generalised Nash inequality]\label{Poincare}
Let $f\in \So$. For all $X\in\RN$ and $t>0$, we have
$$\int_{\RN}{|f(Y)- P_tf(X)|^2p(X,Y,t)\,{\rm d}Y}\leq 2t\int_{\RN}{<K(t)\nabla f(Y),\nabla f(Y)> p(X,Y,t)\,{\rm d}Y}.$$
\end{proposition}

The proof of this result is deferred to subsection \ref{SS:proof}. It should be obvious to the reader that, if we let $\sA = \Delta$ in \eqref{K0}, then from \eqref{Kt} we see that $K(t) \equiv I_N$, and thus the case $t = 1$ and $X=0$ of Proposition \ref{Poincare} is exactly the inequality of Nash \eqref{pg}. 

Interestingly, Proposition \ref{Poincare} implies the following localised Poincar\'e inequality on the intertwined pseudo-balls $B_{r^2}(X,r)$ in \eqref{pb}, see also \eqref{m}. For a given function $f$ and a measurable set $E\subset \RN$, we indicate with $f_E= |E|^{-1} \int_E f(Y) dY$  the average of $f$ on $E$. For ease of notation, in the following statement we let $f_r = f_{B_{r^2}(X,r)}$.

\begin{corollary}\label{C:henri}
Let $r>0$, $X\in\RN$, and $f\in C^\infty_0(B_{4r^2}(X,2r))$. Then, with $C = 2 e^{1/4}$, one has
\[
\int_{B_{r^2}(X,r)} |f(Y) - f_{r}|^2 dY \le C r^2 \int_{B_{4r^2}(X,2r)} <K(r^2) \nabla f(Y),\nabla f(Y)> dY.
\]
\end{corollary} 

\begin{proof}
Since we obviously have $f\in \So$, with the choice $t = r^2$ we can apply Proposition \ref{Poincare} to $f$, obtaining
\[
\int_{\RN}{|f(Y)- P_{r^2} f(X)|^2p(X,Y,r^2)\,{\rm d}Y}\leq 2 r^2 \int_{\RN}{<K(r^2)\nabla f(Y),\nabla f(Y)> p(X,Y,r^2)\,{\rm d}Y}.
\]
On the other hand, we trivially have
\[
\int_{B_{r^2}(X,r)}{|f(Y)- P_{r^2} f(X)|^2p(X,Y,r^2)\,{\rm d}Y}\leq \int_{\RN}{|f(Y)- P_{r^2} f(X)|^2p(X,Y,r^2)\,{\rm d}Y}.
\]
Now, on the set $B_{r^2}(X,r)$ we have $m_{r^2}(X,Y)^2 \le r^2$, and therefore
\[
p(X,Y,r^2)  = \frac{c_N}{V(r^2)} \exp\left(-\frac{m_{r^2}(X,Y)^2}{4r^2}\right) \ge \frac{c_N}{V(r^2)} \exp(-1/4).
\]
We thus find
\begin{align*}
& \frac{c_N}{V(r^2)} \exp(-1/4) \int_{B_{r^2}(X,r)} |f(Y)- P_{r^2} f(X)|^2 dY
\\
& \le 2 r^2 \int_{\RN}{<K(r^2)\nabla f(Y),\nabla f(Y)> p(X,Y,r^2)\,{\rm d}Y}
\\
& = \frac{2 c_N r^2}{V(r^2)} \int_{B_{4r^2}(X,2r)}{<K(r^2)\nabla f(Y),\nabla f(Y)> \exp\left(-\frac{m_{r^2}(X,Y)^2}{4r^2}\right) \,{\rm d}Y}
\\
& \le \frac{2 c_N r^2}{V(r^2)} \int_{B_{4r^2}(X,2r)}{<K(r^2)\nabla f(Y),\nabla f(Y)>  {\rm d}Y}.
\end{align*}
We thus obtain
\[
\int_{B_{r^2}(X,r)} |f(Y)- P_{r^2} f(X)|^2 dY \le 2 e^{1/4}r^2 \int_{B_{4r^2}(X,2r)}{<K(r^2)\nabla f(Y),\nabla f(Y)>  {\rm d}Y}.
\]
The desired conclusion follows observing that
\[
\int_{B_{r^2}(X,r)} |f(Y)- f_r|^2 dY \le \int_{B_{r^2}(X,r)} |f(Y)- P_{r^2} f(X)|^2 dY.  
\]  
\end{proof}

\begin{remark}\label{R:po}
We note that when $\sA = \Delta$ in \eqref{A0}, then $K(t) \equiv I_N$, and we see from \eqref{m} and \eqref{pb} that $B_{r^2}(X,r) = \{Y\in \RN\mid |Y-X|<r\}$ is the standard Euclidean ball $B(X,r)$. In such situation, Corollary \ref{C:henri} is nothing but the following form of the classical Poincar\'e inequality for $f\in W_0^{1,2}(B(X,2r))$,
\[
\int_{B(X,r)} |f(Y) - f_{B(X,r)}|^2 dY \le C r^2 \int_{B(X,2r)} |\nabla f(Y)|^2 dY.
\]
\end{remark}

\subsection{An inequality of Bakry-\'Emery type}
In this second part of the section we prove Proposition \ref{Poincare}. In preparation for it we establish an inequality reminiscent of one first proved by Bakry-\'Emery. 

\begin{lemma}\label{BE}
Let $f\in\So$. For $X\in \RN$ and $\tau>0$ we have
$$<Q\nabla P_{\tau}f(X),\nabla P_{\tau}f(X)> \leq P_{\tau}\left(<e^{\tau B}Qe^{\tau B^\star}\nabla f,\nabla f>\right)(X).$$
\end{lemma}
\begin{proof}
From the explicit expression of the kernel in \eqref{hs}, one can see that
$$\nabla_X p(X,Y,\tau)=-e^{\tau B^\star}\nabla_Y p(X,Y,\tau) \qquad\mbox{for all }X,Y\in\RN,\,\,\tau>0.$$
Hence, exploiting also $P_t1=1$, we find
\begin{eqnarray*}
&&<Q\nabla P_{\tau}f(X),\nabla P_{\tau}f(X)> = <Q^{1/2}\nabla P_{\tau}f(X),Q^{1/2}\nabla P_{\tau}f(X)>
\\
&=&\sum_{j=1}^N\left(\int_{\RN}{f(Y)\left(Q^{\frac{1}{2}}\nabla_X \right)_j p(X,Y,\tau)\,{\rm d}Y}\right)^2 \\ 
&=&\sum_{j=1}^N\left(-\int_{\RN}{f(Y)\left(Q^{\frac{1}{2}}e^{\tau B^\star}\nabla_Y \right)_j p(X,Y,\tau)\,{\rm d}Y}\right)^2 \\
&=&\sum_{j=1}^N\left(\int_{\RN}{\left(Q^{\frac{1}{2}}e^{\tau B^\star}\nabla f(Y) \right)_j p(X,Y,\tau)\,{\rm d}Y}\right)^2 \\
&\leq& \sum_{j=1}^N\int_{\RN}{\left(Q^{\frac{1}{2}}e^{\tau B^\star}\nabla f(Y) \right)^2_j p(X,Y,\tau)\,{\rm d}Y}\int_{\RN}{p(X,Y,\tau)\,{\rm d}Y}\\
&=&\int_{\RN}{<e^{\tau B}Qe^{\tau B^\star}\nabla f(Y),\nabla f(Y)> p(X,Y,\tau)\,{\rm d}Y}=P_{\tau}\left(<e^{\tau B}Qe^{\tau B^\star}\nabla f,\nabla f>\right)(X).
\end{eqnarray*}
\end{proof}

We mention that, in the special case of the Kolmogorov operator $\Delta_v + <v,\nabla_x> - \p_t$ in $\R^{2n+1}$, a more indirect proof of Lemma \ref{BE} can be found in \cite[Proposition 2.5]{BGM}. There, the authors perform a suitable perturbation of a carr\'e du champ related to the operator.

\subsection{Proof of Proposition \ref{Poincare}}\label{SS:proof}
Fix $f\in \So$, $t>0$, and $X\in\RN$. For any $0\leq s \leq t$, we define
$$\psi(s)=P_s\left((P_{t-s}f)^2\right)(X).$$
Notice that
$$\psi(0)=(P_{t}f(X))^2\qquad \mbox{ and }\qquad \psi(t)=P_{t}\left(f^2\right)(X).$$
The reason for introducing the function $\psi(s)$ is that
$$\int_{\RN}{|f(Y)- P_tf(X)|^2p(X,Y,t)\,{\rm d}Y}=P_{t}\left(f^2\right)(X)-(P_{t}f(X))^2=\psi(t)-\psi(0) = \int_0^t \psi'(s) ds.$$
Therefore, the proof of Proposition \ref{Poincare} is completed if we can show that
\begin{align}\label{be}
\int_0^t \psi'(s) ds &\le 2t\int_{\RN}{<K(t)\nabla f(Y),\nabla f(Y)> p(X,Y,t)\,{\rm d}Y} \\
&= 2 P_t(<tK(t)\nabla f,\nabla f>)(X).
\notag
\end{align}
With this objective in mind, from the chain rule and from \cite[Lemma 2.2]{GT} we obtain for every $0<s<t$
\begin{align*}
\psi'(s)&=\sA P_s\left((P_{t-s}f)^2\right)(X) - 2P_s\left(P_{t-s}f\sA\left(P_{t-s}f\right)\right)(X)\\
&=P_s\left(\sA\left((P_{t-s}f)^2\right)\right)(X)- 2P_s\left(P_{t-s}f\sA\left(P_{t-s}f\right)\right)(X)\\
&=2 P_s\left(<Q\nabla P_{t-s}f,\nabla P_{t-s}f>\right)(X) + 2 P_s\left(P_{t-s}f \sA\left(P_{t-s}f\right)\right)(X)\\
&- 2P_s\left(P_{t-s}f\sA\left(P_{t-s}f\right)\right)(X)\\
&=2 P_s\left(<Q\nabla P_{t-s}f,\nabla P_{t-s}f>\right)(X).
\end{align*}
By Lemma \ref{BE} and the semigroup property of $\{P_t\}$, we thus find
\begin{align*}
\psi'(s)&\leq 2 P_s( P_{t-s}(<e^{(t-s) B}Qe^{(t-s) B^\star}\nabla f,\nabla f>))(X)\\
&= 2 P_{t}(<e^{(t-s) B}Qe^{(t-s) B^\star}\nabla f,\nabla f>)(X).
\end{align*} 
This gives
\begin{align*}
& \int_0^{t}\psi'(s)\,{\rm d}s \le 2\int_0^{t}P_{t}(<e^{(t-s) B}Qe^{(t-s) B^\star}\nabla f,\nabla f>)(X)\,{\rm d}s\\
&=2P_{t}\left(\int_0^{t}<e^{(t-s) B}Qe^{(t-s) B^\star}\nabla f,\nabla f>\,{\rm d}s\right)(X)\\
&=2P_{t}\left(<tK(t)\nabla f,\nabla f>\right)(X).
\end{align*}
This proves \eqref{be}, thus completing the proof.

\medskip

We close this section by noting that a full-strength analogue of the classical Poincar\'e inequality for the operator $\sA$ seems to be presently missing. In this respect, we mention the work \cite{WZ}, where the authors prove localised Poincar\'e inequalities for subsolutions of a class of divergence form equations with bounded measurable coefficients as in \eqref{divkolmo} below. In their recent work \cite{AM} the authors establish a Poincar\'e inequality on unbounded cylinders in suitable mixed Sobolev-Gaussian spaces adapted to the special operator $\sA = \Delta_v - <v,\nabla_v> + <v,\nabla_x>$ in $\R^{n}\times\R^n$.
We also mention that in \cite{GIMV} the authors obtain a Moser-type Harnack inequality for nonnegative solutions to
\begin{equation}\label{divkolmo}
{\rm div}_v\left(A(v,x,t)\nabla_v u\right)+<v,\nabla_xu>-\de_tu=0, 
\end{equation}
where $(v,x,t)\in\R^n\times\R^n\times\R$. The matrix $A(\cdot)$ is uniformly positive definite with bounded measurable entries. Their method does not make use of adapted versions of the Poincar\'e inequality. Exploiting the Harnack inequality in \cite{GIMV}, in \cite{aLPP} the authors prove a Gaussian lower bound of Aronson type (see also \cite{DiP}, and references therein, for Gaussian lower bounds for equations with H\"older coefficients).

 
\section{A sharp Harnack inequality for the extended equation}\label{sec5}

In their celebrated work  \cite{LY} Li and Yau proved (among other things) that if $f>0$ is a solution of the heat equation $\p_t f - \Delta f = 0$ on a boundaryless, complete $n$-dimensional Riemannan manifold $\mathbb M$ having $\operatorname{Ricci} \ge 0$, then the function $u = \log f$ satisfies the inequality on $\mathbb M\times (0,\infty)$,
\begin{equation}\label{LY0}
|\nabla u|^2 - \p_t u \le \frac{n}{2t}.
\end{equation}
Such inequality becomes an equality when $f$ is the heat kernel in flat $\Rn$. The relevance of \eqref{LY0} is underscored by the fact that a remarkable consequence of it is the following sharp form of the Harnack inequality, valid for any $x, y\in \mathbb M$ and any $0<s<t<\infty$,
$$f(x,s) \le f(y,t) \left(\frac ts\right)^{\frac n2} \exp\left(\frac{d(x,y)^2}{4t}\right).$$

The aim of this section is to prove a related sharp Harnack inequality for the semigroup associated with the extension operator 
$$\K_a = z^a(\K + \Ba),$$
where $\Ba$ is the Bessel operator $\Ba = \frac{\p^2}{\p z^2} + \frac az \frac{\p}{\p z}$ with $a> -1$. The operator $\K_a$ has been introduced in \cite[Section 3]{GT} for a generalisation of the Caffarelli-Silvestre extension result \cite{CS} relatively to the equation \eqref{K0} (see \cite[Theorems 4.1 and 4.2]{GT}). To solve the extension problem for $(-\mathscr K)^s$ we relied on the explicit construction of generalised Poisson kernels (\cite[Definition 3.7]{GT}) via the knowledge of the following Neumann fundamental solution for $\K_a$
\begin{equation}\label{P:Ga}
\mathscr G^{(a)}(X,t,z;Y,\tau,\zeta) = p(X,Y,t-\tau) p^{(a)}(z,\zeta,t-\tau), \quad\mbox{ for }X,Y\in\RN,\, t,\tau\in\R,\,z,\zeta>0,
\end{equation}
where
$$p^{(a)}(z,\zeta,t)  =(2t)^{-\frac{a+1}{2}}\left(\frac{z\zeta}{2t}\right)^{\frac{1-a}{2}}I_{\frac{a-1}{2}}\left(\frac{z\zeta}{2t}\right)e^{-\frac{z^2+\zeta^2}{4t}}$$
(we refer the reader to \cite[Proposition 3.5]{GT}).

For a given $\vf\in C^{\infty}_0(\Rnp)$, we now introduce the \emph{extension semigroup}
\begin{equation}\label{howtodef}
{\mathscr P}^{(a)}_t\vf (X,z)=\int_0^{\infty}{\int_{\RN}{\mathscr G^{(a)}(X,t,z;Y,0,\zeta)\vf(Y,\zeta)\zeta^a\,{\rm d}Y}\,{\rm d}\zeta}, \qquad (X,z)\in\Rnp,\,  t>0.
\end{equation}
The function $u(X,t,z)={\mathscr P}^{(a)}_t\vf(X,z)$ in \eqref{howtodef} solves the Cauchy problem with Neumann condition
$$\begin{cases}
\K_a u = 0\ \  \text{in}\ \Rnp\times (0,\infty),
\\
u(X,0,z) = \vf(X,z)\ \   (X,z)\in\Rnp,
\\
\lim_{z\rightarrow 0^+}{z^a \p_z u(X,t,z)}=0.
\end{cases}$$
The fact that $\{{\mathscr P}^{(a)}_t\}_{t>0}$ defines a stochastically complete semigroup follows from \cite[Proposition 3.6]{GT}.

We note explicitly that, when $\vf(X,z) = \vf(X)$, i.e., the initial datum is independent of $z$, then for every $z>0$, $X\in \RN$ and $t>0$ we have
\[
{\mathscr P}^{(a)}_t\vf (X,z)= P_t \vf(X).
\]
This follows from the observation that when $\vf$ is independent of $z$, then Fubini's theorem gives
\[
{\mathscr P}^{(a)}_t\vf (X,z) = P_t \vf(X) \int_0^\infty p^{(a)}(z,\zeta,t) \zeta^a d\zeta = P_t \vf(X), 
\]
where in the last equality we have used \cite[Proposition 2.3]{G18c}.

In Theorem \ref{shth} below we are going to establish a global Harnack estimate for ${\mathscr P}^{(a)}_t\vf$, with $\vf\geq 0$. The key step is a remarkable Li-Yau type inequality satisfied by the semigroup $\{{\mathscr P}^{(a)}_t\}_{t>0}$, see Lemma \ref{LYKaPt} below. We start with the following preliminary lemma, where we compute the derivative of the function $V(t)$ in \eqref{VS}. We denote
\begin{equation}\label{Ct}
C(t)=\int_0^t e^{-sB} Q e^{-s B^\star} ds > 0\qquad\mbox{for any }t>0.
\end{equation}
Using this notation, it is known (see e.g. \cite{LP}) that the kernel $p(X,Y,t)$ in \eqref{hs} reads as
\begin{equation}\label{fundholp}
p(X,Y,t)=\dfrac{(4\pi)^{-\frac{N}{2}}e^{-t\operatorname{tr}B}}{\sqrt{\det(C(t))}} \exp{ \left(  -\dfrac{\langle C^{-1}(t) \left(X-e^{-tB}Y\right), X-e^{-tB}Y \rangle}{4}   \right)}.
\end{equation}

\begin{lemma}
For all $t>0$ we have
\begin{equation}\label{relogd}
\operatorname{tr}\left(Q C^{-1}(t)\right)=\frac{d}{dt}\left(\log{\left( \operatorname{det}\left(C(t)\right)\right)}\right) + 2\operatorname{tr}(B)=\frac{d}{dt}\left(\log{ \left(\operatorname{det}\left(tK(t)\right)\right)}\right).
\end{equation}
\end{lemma}
\begin{proof}
From the relation $tK(t)=e^{tB}C(t)e^{tB^\star}$ we deduce that for every $t>0$,
$$\operatorname{det}\left(tK(t)\right)=e^{2t \operatorname{tr}(B)}\operatorname{det}\left(C(t)\right),$$
which easily implies the second equality in \eqref{relogd}. Concerning the first equality, we recall the formula 
\[
\frac{d}{dt}(\operatorname{det}(M(t)))=\operatorname{tr}(M'(t)M^{-1}(t))\,\operatorname{det}(M(t)),
\]
 which holds true for any symmetric invertible matrix $M(t)$, and gives
\begin{equation}\label{flogdet}
\frac{d}{dt}\left(\log{\left( \operatorname{det}\left(C(t)\right)\right)}\right)=\operatorname{tr}(C'(t)C^{-1}(t)).
\end{equation}
On the other hand, from the definition \eqref{Ct}, we obtain $C'(t)=e^{-tB}Q e^{-tB^*}$. We also know (see e.g. \cite[equation (4.6)]{AT}) that
\begin{equation}\label{dinuovodert}
e^{-tB}Q e^{-tB^*}= Q  - B C(t) - C(t) B^*.
\end{equation}
Inserting this in \eqref{flogdet}, we finally have
$$\frac{d}{dt}\left(\log{\left( \operatorname{det}\left(C(t)\right)\right)}\right)=\operatorname{tr}(\left( Q  - B C(t) - C(t) B^*\right)C^{-1}(t)) = \operatorname{tr}\left(Q C^{-1}(t)\right) - 2\operatorname{tr}(B).$$
\end{proof}

The following two lemmas are the crucial Li-Yau type estimates respectively for the Neumann fundamental solution $\mathscr G^{(a)}$ and for the nonnegative solutions $\mathscr{P}^{(a)}_t\vf$.

\begin{lemma}\label{LYKa}
For any $X,Y\in\RN$, $z,\zeta >0$, $t>\tau$, we denote
$$u(X,t,z)=\log\ \mathscr G^{(a)}(X,t,z;Y,\tau,\zeta).$$
Then, for $a\geq 0$ we have
\begin{align}\label{LiYau-GKa}
& \left\langle Q\nabla_X u (X,t,z),\nabla_X u(X,t,z)\right\rangle + (\de_z u (X,t,z))^2 + \left\langle BX,\nabla_X u(X,t,z)\right\rangle -\de_t u(X,t,z)
\\
&  <\frac{1}{2}\operatorname{tr}\left(Q C^{-1}(t-\tau)\right)+\frac{a+1}{2(t-\tau)}.
\notag
\end{align}
If instead $z = 0$, then \eqref{LiYau-GKa} holds true for any $a>-1$.
\end{lemma}

\begin{proof}
Recalling \eqref{P:Ga}, we have
\begin{align*}
& \left\langle Q\nabla_X u (X,t,z),\nabla_X u(X,t,z)\right\rangle + (\de_z u (X,t,z))^2 + \left\langle BX,\nabla_X u(X,t,z)\right\rangle -\de_t u(X,t,z)
\\
& = \left\langle Q\nabla_X \log p(X,Y,t-\tau),\nabla_X \log p(X,Y,t-\tau)\right\rangle + \left\langle BX,\nabla_X \log p(X,Y,t-\tau)\right\rangle 
\\
& -\de_t \log p(X,Y,t-\tau) + (\de_z \log p^{(a)}(z,\zeta,t-\tau))^2-\de_t \log p^{(a)}(z,\zeta,t-\tau).
\end{align*}
Moreover, from \cite[equation (4.4)]{GT}, we have
\begin{equation}\label{Xderp2}
\nabla_X \log p(X,Y,t-\tau)=-\frac{1}{2}C^{-1}(t-\tau)\left(X-e^{-(t-\tau)B}Y\right).
\end{equation}
Furthermore, \eqref{fundholp} gives
\begin{eqnarray}\label{detlog}
&&\de_t \log p(X,Y,t-\tau) = -\operatorname{tr}(B)-\frac{1}{2}\frac{d}{dt}\left(\log{\left( \operatorname{det}\left(C(t)\right)\right)}\right)\\
&+&\frac{1}{4}\left\langle C'(t-\tau)C^{-1}(t-\tau) \left(X-e^{-(t-\tau)B}Y\right),C^{-1}(t-\tau) \left(X-e^{-(t-\tau)B}Y\right)\right\rangle\notag\\
&-&\frac{1}{2}\left\langle Be^{-(t-\tau)B}Y,C^{-1}(t-\tau) \left(X-e^{-(t-\tau)B}Y\right)\right\rangle\notag\\
&=&-\frac{1}{2}\operatorname{tr}\left(Q C^{-1}(t-\tau)\right)-\frac{1}{2}\left\langle BX, C^{-1}(t-\tau)\left(X-e^{-(t-\tau)B}Y\right)\right\rangle\notag\\
&+&\frac{1}{4}\left\langle QC^{-1}(t-\tau) \left(X-e^{-(t-\tau)B}Y\right),C^{-1}(t-\tau) \left(X-e^{-(t-\tau)B}Y\right)\right\rangle,
\notag\end{eqnarray}
where in the last equality we have used \eqref{relogd} and \eqref{dinuovodert}.
From \eqref{Xderp2} and \eqref{detlog} we deduce
\begin{eqnarray}\label{LY:K}
&&\left\langle Q\nabla_X \log p(X,Y,t-\tau),\nabla_X \log p(X,Y,t-\tau)\right\rangle + \\
&&+\left\langle BX,\nabla_X \log p(X,Y,t-\tau)\right\rangle -\de_t \log p(X,Y,t-\tau)=\frac{1}{2}\operatorname{tr}\left(Q C^{-1}(t-\tau)\right).\nonumber
\end{eqnarray}
We mention that the equation \eqref{LY:K} was first established in the proof of \cite[Proposition 6]{CPP}.
On the other hand, it is proved in \cite[Proposition 4.2]{G18c} that, for $a\geq 0$, one has for any $z, \zeta>0$ and $\tau<t$,
\begin{equation}\label{LY:Ba}
(\de_z \log p^{(a)}(z,\zeta,t-\tau))^2-\de_t \log p^{(a)}(z,\zeta,t-\tau)<\frac{a+1}{2(t-\tau)}.
\end{equation}
Adding \eqref{LY:K} and \eqref{LY:Ba} we obtain \eqref{LiYau-GKa} as desired.

The final statement of the theorem follows from the observation that when $z = 0$, then in \cite[Proposition 4.1, (4.3)]{G18c} it is shown that for any $a>-1$, one has
\begin{equation}\label{LY:Ba22}
(\de_z \log p^{(a)}(0,\zeta,t-\tau))^2-\de_t \log p^{(a)}(0,\zeta,t-\tau)<\frac{a+1}{2(t-\tau)},
\end{equation}
for any $\zeta>0$ and $\tau<t$. 
If we now add \eqref{LY:K} and \eqref{LY:Ba22} we obtain the sought for conclusion \eqref{LiYau-GKa} with $z = 0$, for any $a>-1$.
\end{proof}

\begin{proposition}[Li-Yau inequality for $\mathscr P^{(a)}_t$]\label{LYKaPt}
Let $\vf\in C_0^{\infty}(\Rnp)$ be such that $\vf\geq 0$ and not identically vanishing. For all $X\in\RN$, $t>0$, and $z>0$, we have for every $a\ge 0$, 
\begin{align}\label{LiYau-Pta}
& \left\langle Q\nabla_X \log{\mathscr P}^{(a)}_t\vf (X,z) ,\nabla_X \log{\mathscr P}^{(a)}_t\vf (X,z)\right\rangle + (\de_z \log{\mathscr P}^{(a)}_t\vf (X,z))^2 \\
& + \left\langle BX,\nabla_X \log{\mathscr P}^{(a)}_t\vf (X,z)\right\rangle -\de_t \log{\mathscr P}^{(a)}_t\vf (X,z) < \frac{1}{2}\operatorname{tr}\left(Q C^{-1}(t)\right)+\frac{a+1}{2t}.
\notag
\end{align}
If, instead, $z = 0$, then the inequality \eqref{LiYau-Pta} continue to be valid for any $a>-1$.
\end{proposition}

\begin{proof}
Since $\vf \in C_0^{\infty}(\Rnp)$, we can differentiate ${\mathscr P}^{(a)}_t\vf (X,z)$ under the integral sign around any $(X,t,z)\in \Rnp\times(0,\infty)$. This gives
\begin{eqnarray*}
&&\left\langle Q\nabla_X {\mathscr P}^{(a)}_t\vf (X,z) ,\nabla_X {\mathscr P}^{(a)}_t\vf (X,z)\right\rangle + (\de_z {\mathscr P}^{(a)}_t\vf (X,z))^2 \\
&=&\sum_{j=1}^N\left(\int_{\RN\times\R^+}{\vf(Y,\zeta)\zeta^a\left(Q^{\frac{1}{2}}\nabla_X\right)_j\mathscr G^{(a)}(X,t,z;Y,0,\zeta)\,{\rm d}Y\,{\rm d}\zeta}\right)^2 \\ 
&+&\left(\int_{\RN\times\R^+}{\vf(Y,\zeta)\zeta^a\de_z\mathscr G^{(a)}(X,t,z;Y,0,\zeta)\,{\rm d}Y\,{\rm d}\zeta}\right)^2
\\
&\leq& {\mathscr P}^{(a)}_t\vf (X,z)\int_{\RN\times\R^+}{\vf(Y,\zeta)\zeta^a\frac{\left\langle Q\nabla_X \mathscr G^{(a)}(X,t,z;Y,0,\zeta) ,\nabla_X \mathscr G^{(a)}(X,t,z;Y,0,\zeta)\right\rangle}{\mathscr G^{(a)}(X,t,z;Y,0,\zeta)}\,{\rm d}Y\,{\rm d}\zeta}\\
&+&{\mathscr P}^{(a)}_t\vf (X,z)\int_{\RN\times\R^+}{\vf(Y,\zeta)\zeta^a\frac{\left(\de_z \mathscr G^{(a)}(X,t,z;Y,0,\zeta)\right)^2}{\mathscr G^{(a)}(X,t,z;Y,0,\zeta)}\,{\rm d}Y\,{\rm d}\zeta}\\
&<& \left(\frac{1}{2}\operatorname{tr}\left(Q C^{-1}(t)\right)+\frac{a+1}{2t}\right)\left({\mathscr P}^{(a)}_t\vf (X,z)\right)^2 - {\mathscr P}^{(a)}_t\vf (X,z) \left\langle BX,\nabla_X {\mathscr P}^{(a)}_t\vf (X,z)\right\rangle \\
&+&{\mathscr P}^{(a)}_t\vf (X,z) \de_t {\mathscr P}^{(a)}_t\vf (X,z).
\end{eqnarray*}
We note that, in the last inequality, the Li-Yau type inequality \eqref{LiYau-GKa} of the previous lemma is used in a crucial way. The inequality \eqref{LiYau-Pta} is now obtained by rearranging terms, and dividing by $\left({\mathscr P}^{(a)}_t\vf (X,z)\right)^2$.

The second part of the statement of the proposition follows in a similar fashion by appealing to the second part of Lemma \ref{LYKa}.

\end{proof}

We are now ready to prove the desired Harnack inequality.

\begin{theorem}[Sharp Harnack inequality]\label{shth} Let $a \geq 0$, and let $\vf\geq 0$ such that $\vf\in C_0^{\infty}(\Rnp)$. For $X,Y\in\RN, z,\zeta>0$ and $0 < s < t < \infty$, we have
\begin{eqnarray}\label{harnack}
{\mathscr P}^{(a)}_s\vf (Y,\zeta)\leq {\mathscr P}^{(a)}_t\vf (X,z) \left(\frac{t}{s}\right)^{\frac{N+a+1}{2}}\left(\frac{\operatorname{det}\left(K(t)\right)}{\operatorname{det}\left(K(s)\right)}\right)^{\frac{1}{2}}\exp\left(\frac{|z-\zeta|^2}{4(t-s)}\right)\cdot&& \\
\cdot\exp\left(\frac{1}{4}\left\langle C^{-1}(t-s)\left(X-e^{-(t-s)B}Y\right),\left(X-e^{-(t-s)B}Y\right)\right\rangle\right).&& \nonumber
\end{eqnarray}
When $z=\zeta=0$ the inequality is valid for every $a>-1$, and reads 
\begin{eqnarray}\label{harnack0}
{\mathscr P}^{(a)}_s\vf (Y,0)\leq {\mathscr P}^{(a)}_t\vf (X,0) \left(\frac{t}{s}\right)^{\frac{N+a+1}{2}}\left(\frac{\operatorname{det}\left(K(t)\right)}{\operatorname{det}\left(K(s)\right)}\right)^{\frac{1}{2}}\cdot&& \\
\cdot\exp\left(\frac{1}{4}\left\langle C^{-1}(t-s)\left(X-e^{-(t-s)B}Y\right),\left(X-e^{-(t-s)B}Y\right)\right\rangle\right).&& \nonumber
\end{eqnarray}
\end{theorem}
\begin{proof}
We can assume $\vf\not\equiv 0$ (otherwise we have nothing to prove), and denote $u(X,t,z)= {\mathscr P}^{(a)}_t\vf (X,z)$. Let us fix $X,Y\in\RN, z,\zeta>0$ and $0 < s < t < \infty$. We are going to choose an optimal curve joining $(X,t,z)$ and $(Y,s,\zeta)$. Let us consider the curve
$$\eta(\tau)=\left(\gamma(\tau),t-\tau, z-\frac{\tau}{t-s}(z-\zeta)\right),\qquad\mbox{for }\tau\in[0,t-s],$$
where $\gamma$ is the smooth curve in $\RN$ defined by
$$
\gamma(\tau)=e^{\tau B}\left(X-C(\tau)C^{-1}(t-s)\left( X-e^{-(t-s)B}Y \right)\right).
$$
We then have
$$\eta(0)=(X,t,z),\qquad\eta(t-s)=(Y,s,\zeta),\qquad\eta(\tau)\in \RN\times\R^+\times\R^+ \quad\forall\tau\in[0,t-s].$$
Moreover, if we denote by
\begin{equation}\label{defomega}
\omega(\tau)=-Q^{\frac{1}{2}}e^{-\tau B^\star}C^{-1}(t-s)\left( X-e^{-(t-s)B}Y \right)
\end{equation}
and recalling that $C'(\tau)=e^{-\tau B}Qe^{-\tau B^\star}$, we have
\begin{equation}\label{etaprimo}
\eta'(\tau)=\left( B\gamma(\tau) + Q^{\frac{1}{2}}\omega(\tau), -1, \frac{\zeta-z}{t-s}\right).
\end{equation}
We also note that, from the definitions \eqref{defomega} and \eqref{Ct}, we find
\begin{eqnarray}\label{costomega}
&&\int_0^{t-s}|\omega(\tau)|^2\,{\rm d}\tau\\
&=&\int_0^{t-s}\left\langle e^{-\tau B}Qe^{-\tau B^\star} C^{-1}(t-s)\left( X-e^{-(t-s)B}Y \right), C^{-1}(t-s)\left( X-e^{-(t-s)B}Y \right) \right\rangle\,{\rm d}\tau \nonumber \\
&=& \left\langle C^{-1}(t-s)\left(X-e^{-(t-s)B}Y\right),\left(X-e^{-(t-s)B}Y\right)\right\rangle.\nonumber
\end{eqnarray}
The optimality we have claimed for the curve $\gamma$ consists in the following: among the curves in $[0,t-s]$ joining $X$ and $Y$ which are admissible for the control problem related to $\K$ (i.e. $\gamma'(\tau)=B\gamma(\tau) + Q^{\frac{1}{2}}\tilde{\omega}(\tau)$), the particular one we have chosen minimizes the cost $\int_0^{t-s}|\tilde{\omega}(\tau)|^2\,{\rm d}\tau$. We refer the reader to \cite[Section 6]{DiP} and \cite{BP} (and references therein) for further details.
\\
Denoting $h(\tau)=\log u(\eta(\tau))$ and using \eqref{etaprimo}, we then obtain
\begin{eqnarray*}
&&\log\frac{u(Y,s,\zeta)}{u(X,t,z)}=h(t-s)-h(0)=\int_0^{t-s}h'(\tau)\,{\rm d}\tau\\
&=&\int_0^{t-s}\left\langle \nabla_X\log u(\eta(\tau)), \gamma'(\tau)\right\rangle\,{\rm d}\tau - \int_0^{t-s}\de_t\log u(\eta(\tau))\,{\rm d}\tau + \frac{\zeta-z}{t-s}\int_0^{t-s}\de_z\log u(\eta(\tau))\,{\rm d}\tau\\
&=&\int_0^{t-s}\left\langle Q^{\frac{1}{2}}\nabla_X\log u(\eta(\tau)), \omega(\tau)\right\rangle\,{\rm d}\tau + \int_0^{t-s}\left\langle B\gamma(\tau),\nabla_X\log u(\eta(\tau))\right\rangle\,{\rm d}\tau\\
&-& \int_0^{t-s}\de_t\log u(\eta(\tau))\,{\rm d}\tau + \frac{\zeta-z}{t-s}\int_0^{t-s}\de_z\log u(\eta(\tau))\,{\rm d}\tau\\
&\leq& \left(\int_0^{t-s}|\omega(\tau)|^2\,{\rm d}\tau\right)^{\frac{1}{2}}\left(\int_0^{t-s}\left\langle Q\nabla_X\log u(\eta(\tau)), \nabla_X\log u(\eta(\tau))\right\rangle\,{\rm d}\tau\right)^{\frac{1}{2}}\\
&+& \frac{|z-\zeta|}{\sqrt{t-s}}\left(\int_0^{t-s}\left(\de_z\log u(\eta(\tau))\right)^2\,{\rm d}\tau\right)^{\frac{1}{2}} \\
&+&\int_0^{t-s}\left\langle B\gamma(\tau),\nabla_X\log u(\eta(\tau))\right\rangle\,{\rm d}\tau - \int_0^{t-s}\de_t\log u(\eta(\tau))\,{\rm d}\tau \\
&\leq& \frac{1}{4}\int_0^{t-s}|\omega(\tau)|^2\,{\rm d}\tau + \frac{1}{4}\frac{|z-\zeta|^2}{t-s}+ \int_0^{t-s}\left\langle Q\nabla_X\log u(\eta(\tau)), \nabla_X\log u(\eta(\tau))\right\rangle\,{\rm d}\tau \\
&+& \int_0^{t-s}\left(\de_z\log u(\eta(\tau))\right)^2\,{\rm d}\tau +\int_0^{t-s}\left\langle B\gamma(\tau),\nabla_X\log u(\eta(\tau))\right\rangle\,{\rm d}\tau - \int_0^{t-s}\de_t\log u(\eta(\tau))\,{\rm d}\tau.
\end{eqnarray*}
We are now in position to apply \eqref{LiYau-Pta} (computed in fact at $\eta(\tau)$) and to deduce, by using also \eqref{costomega} and \eqref{relogd}, the following
\begin{eqnarray*}
&&\log\frac{u(Y,s,\zeta)}{u(X,t,z)}\leq \frac{1}{4}\int_0^{t-s}|\omega(\tau)|^2\,{\rm d}\tau + \frac{1}{4}\frac{|z-\zeta|^2}{t-s}+ \frac{1}{2}\int_0^{t-s}\left(\operatorname{tr}\left(Q C^{-1}(t-\tau)\right)+\frac{a+1}{t-\tau}\right)\,{\rm d}\tau \\
&=&\frac{1}{4}\left\langle C^{-1}(t-s)\left(X-e^{-(t-s)B}Y\right),\left(X-e^{-(t-s)B}Y\right)\right\rangle + \frac{1}{4}\frac{|z-\zeta|^2}{t-s}\\
&+&\frac{1}{2}\int_0^{t-s}\left(-\frac{d}{d\tau}\left(\log{ \left(\operatorname{det}\left((t-\tau)K_{t-\tau}\right)\right)}\right)+\frac{a+1}{t-\tau}\right)\,{\rm d}\tau\\
&=&\frac{1}{4}\left\langle C^{-1}(t-s)\left(X-e^{-(t-s)B}Y\right),\left(X-e^{-(t-s)B}Y\right)\right\rangle + \frac{1}{4}\frac{|z-\zeta|^2}{t-s}\\
&+&\frac{1}{2}\log{\left( \frac{\operatorname{det}\left(tK(t)\right)}{\operatorname{det}\left(sK(s)\right)}\right)}+\frac{a+1}{2}\log{\left( \frac{t}{s}\right)}.
\end{eqnarray*}
Noticing that $\operatorname{det}\left(tK(t)\right)=t^N\operatorname{det}\left(K(t)\right)$ and exponentiating both side of the previous inequality, we finally reach \eqref{harnack}.\\
In order to prove \eqref{harnack0}, we just mention that the curve $\eta(\tau)$ becomes $(\gamma(\tau),t-\tau,0)$ since $z=\zeta=0$. Following the proof of the first part, we thus need to apply \eqref{LiYau-Pta} computed only at $z=0$: that this can be done for any $a>-1$ follows from the second part of Proposition \ref{LYKaPt}.

\end{proof}

\begin{remark}
We would like to comment about the sharpness of the Harnack estimate in Theorem \ref{shth}. For $\ve>0$, consider the function
$$u(X,t,z)=\mathscr G^{(a)}(X,t,z;0,-\ve,0).$$
For $0<s<t$, $X\in\RN$, and $z,\zeta\in\R^+$, by definition we have
$$\frac{u(0,s,\zeta)}{u(X,t,z)}=\left(\frac{t+\ve}{s+\ve}\right)^{\frac{N+a+1}{2}}e^{\frac{z^2}{4(t+\ve)}-\frac{\zeta^2}{4(s+\ve)}}\left(\frac{\operatorname{det}\left(K(t+\ve)\right)}{\operatorname{det}\left(K(s+\ve)\right)}\right)^{\frac{1}{2}}\exp\left(\frac{1}{4}\left\langle C^{-1}(t+\ve)X,X\right\rangle\right).$$
As $z,\zeta,\ve$ tend to $0^+$, the expression in the right-hand side approaches the bound in \eqref{harnack}. We remark that, in view of \cite[Proposition 3.6]{GT}, we can write
$$u(X,t,z)={\mathscr P}^{(a)}_t\left(\mathscr G^{(a)}(\cdot,\ve,\cdot;0,0,0)\right) (X,z).$$
We note that, even if not in $C_0^{\infty}(\Rnp)$, the function $\mathscr G^{(a)}(X,\ve,z;0,0,0)$ can be monotonically approximated with suitable cut-off functions whose supports exhaust the whole space.
\end{remark}


\bibliographystyle{amsplain}

\begin{thebibliography}{10}

\bibitem{AV}
N. Abatangelo \& E. Valdinoci, \emph{Getting acquainted with the fractional Laplacian}. In "Contemporary Research in Elliptic PDEs and Related Topics", to appear in Springer INdAM Series \textbf{33}~(2019), \verb|www.springer.com/gp/book/9783030189204#aboutBook|

\bibitem{AT} F. Abedin \& G. Tralli, 
\textit{Harnack inequality for a class of Kolmogorov-Fokker-Planck equations in non-divergence form}. 
Arch. Rational Mech. Anal. \textbf{233}~(2019), 867-900. 

\bibitem{AM}
S. Armstrong \& J.-C. Mourrat, \emph{Variational methods for the kinetic Fokker-Planck equation}. ArXiv preprint 1902.04037.

\bibitem{A68}
D.G. Aronson, \emph{Non-negative solutions of linear parabolic equations}. Ann. Scuola Norm. Sup. Pisa (3) \textbf{22}~(1968), 607-694. 

\bibitem{A}
N. Aronszajn, \emph{Boundary values of functions with finite Dirichlet integral}. Techn. Report of Univ. of Kansas 14 (1955), 77-94.

\bibitem{B59}
A.V. Balakrishnan, \emph{An operational calculus for infinitesimal generators of semigroups}. Trans. Amer. Math. Soc. \textbf{91}~(1959), 330-353. 

\bibitem{B}
A.V. Balakrishnan, 
\textit{Fractional powers of closed operators and the semigroups generated by them}. 
Pacific J. Math. \textbf{10}~(1960), 419--437.

\bibitem{BP}
D. Barilari \& E. Paoli, \emph{Curvature terms in small time heat kernel expansion for a model class of hypoelliptic H\"ormander operators}. Nonlinear Anal. \textbf{164}~(2017), 118-134. 

\bibitem{BGM}
F. Baudoin, M. Gordina \& P. Mariano, \emph{Gradient bounds for Kolmogorov type diffusions}. ArXiv preprint 1803.01436.

\bibitem{BC}
S. Bochner \& K. Chandrasekharan, \emph{Fourier Transforms}. Annals of Mathematics Studies, no. \textbf{19}, Princeton University Press, Princeton, N. J.; Oxford University Press, London, 1949. ix+219 pp.

\bibitem{BV}
C. Bucur \& E. Valdinoci, \emph{Nonlocal Diffusion and Applications}. Lecture Notes of the Unione Matematica Italiana, Springer, 2016. 

\bibitem{CRS}
L. Caffarelli, J.-M. Roquejoffre \& O. Savin, \emph{Nonlocal minimal surfaces}. Comm. Pure Appl. Math. \textbf{63}~(2010), no. 9, 1111-1144.

\bibitem{CS}
L. Caffarelli \& L. Silvestre, \emph{An extension problem related to the fractional Laplacean}. Comm. Partial Differential Equations \textbf{32}~(2007), no. 7-9, 1245-1260.

\bibitem{CPP}
A. Carciola, A. Pascucci \& S. Polidoro, \emph{Harnack inequality and no-arbitrage bounds for self-financing portfolios}. Bol. Soc. Esp. Mat. Apl. SeMA \textbf{49}~(2009), 19-31.

\bibitem{Ch} 
S. Chandrasekhar, 
\textit{Stochastic Problems in Physics and Astronomy}. Rev. Mod. Phys. \textbf{15}~(1943), 1-89.

\bibitem{DGN}
D. Danielli, N. Garofalo \& D.M. Nhieu, \emph{Non-doubling Ahlfors measures, perimeter measures, and the characterization of the trace spaces of Sobolev functions in Carnot-Carath\'eodory spaces}. Mem. Amer. Math. Soc. \textbf{182}~ (2006), no. 857, x+119 pp. 

\bibitem{DiP} 
M. Di Francesco \& S. Polidoro,
\textit{Schauder estimates, Harnack inequality and Gaussian lower bound for Kolmogorov-type operators in non-divergence form}. 
Adv. Differential Equations 11 (2006), 1261--1320.

\bibitem{EN}
K.-J. Engel \& R. Nagel, \emph{A short course on operator semigroups}. Universitext. Springer, New York, 2006. x+247 pp.

\bibitem{FS}
E.B. Fabes \& D.W. Stroock, \emph{A new proof of Moser's parabolic Harnack inequality using the old ideas of Nash}. Arch. Rational Mech. Anal. \textbf{96}~(1986), no. 4, 327-338. 

\bibitem{Ga}
E. Gagliardo, \emph{Caratterizzazioni delle tracce sulla frontiera relative ad alcune classi di funzioni in $n$ variabili}. (Italian) Rend. Sem. Mat. Univ. Padova \textbf{27}~(1957), 284-305. 

\bibitem{Gft}
N. Garofalo,
\textit{Fractional thoughts}.
In ``New Developments in the Analysis of Nonlocal Operators'', Contemp. Math. Vol. 723, 1-135, Amer. Math. Soc., Providence, RI, 2019.

\bibitem{G18c}
N. Garofalo,
\textit{Two classical properties of the Bessel quotient $I_{\nu+1}/I_\nu$ and their implications in pde's}. ArXiv preprint 1810.09756.

\bibitem{GL}
N. Garofalo \& E. Lanconelli, \emph{Level sets of the fundamental solution and Harnack inequality for degenerate equations of Kolmogorov type}. Trans. Amer. Math. Soc. \textbf{321}~(1990), no. 2, 775-792.

\bibitem{GT}
N. Garofalo \& G. Tralli, \emph{A class of nonlocal hypoelliptic operators and their extensions}. ArXiv preprint 1811.02968.

\bibitem{GThls}
N. Garofalo \& G. Tralli, \emph{Hardy-Littlewood-Sobolev inequalities for a class of non-symmetric and non-doubling hypoelliptic semigroups}. ArXiv preprint 1904.12982.

\bibitem{GTiso}
N. Garofalo \& G. Tralli, \emph{Nonlocal isoperimetric inequalities for Kolmogorov-Fokker-Planck operators}, work in preparation.

\bibitem{GIMV}
F. Golse, C. Imbert \& C. Mouhot, A. Vasseur, \emph{Harnack inequality for kinetic Fokker-Planck equations with rough coefficients and application to the Landau equation}. To appear in Ann. Sc. Norm. Super. Pisa Cl. Sci.

\bibitem{GR}
I.S. Gradshteyn \& I.M. Ryzhik, \emph{Tables of integrals, series, and products}. Academic Press, 1980.

\bibitem{Ho}
L. H{\"o}rmander,
\textit{Hypoelliptic second order differential equations}. Acta Math. \textbf{119}~(1967), 147-171.

\bibitem{Il}
A.M. Il'in, \textit{On a class of ultraparabolic equations}. Soviet Math. Doklady (Dokl. Akad. Nauk SSSR) \textbf{5}~(1964), 1673--1676.

\bibitem{JW}
A. Jonsson \& H. Wallin, \emph{Function spaces on subsets of $\Rn$}. Math. Rep. 2 (1984), no. 1, xiv+221 pp.

\bibitem{KLT} 
A.E. Kogoj, E. Lanconelli \& G. Tralli,  
\textit{Wiener-Landis criterion for Kolmogorov-type operators}.
Discrete Contin. Dyn. Syst. \textbf{38}~(2018), 2467-2485.

\bibitem{Kol}
A.N. Kolmogorov,  
\textit{Zuf\"allige Bewegungen (Zur Theorie der Brownschen Bewegung)}. 
Ann. of Math. (2) \textbf{35}~(1934), 116--117.

\bibitem{Ku}
L.P. Kupcov, \emph{The fundamental solutions of a certain class of elliptic-parabolic second order equations}, (Russian) Differencial'nye Uravnenija \textbf{8}~(1972), 1649-1660.

\bibitem{Kwa}
M. Kwa\'snicki, \emph{Ten equivalent definitions of the fractional Laplace operator}. Fract. Calc. Appl. Anal. \textbf{20}~(2017), no. 1, 7-51. 

\bibitem{aLPP} 
A. Lanconelli, A. Pascucci \& S. Polidoro,  
\textit{Gaussian lower bounds for non-homogeneous Kolmogorov equations with measurable coefficients}. 
ArXiv preprint 1704.07307.

\bibitem{LP}
E. Lanconelli \& S. Polidoro, \emph{On a class of hypoelliptic evolution operators}. Partial differential equations, II (Turin, 1993). Rend. Sem. Mat. Univ. Politec. Torino \textbf{52}~(1994), no. 1, 29-63. 

\bibitem{La}
N.S. Landkof, \emph{Foundations of modern potential theory}, translated from the Russian by A. P. Doohovskoy. Die Grundlehren der mathematischen Wissenschaften, Band 180. Springer-Verlag, New York-Heidelberg, 1972. 

\bibitem{LY}
P. Li \& S.T. Yau, \emph{On the parabolic kernel of the
Schr\"odinger operator}. Acta Math., \textbf{156}~(1986), 153-201.

\bibitem{L}
A. Lunardi,
\textit{Schauder estimates for a class of degenerate elliptic and parabolic operators with unbounded coefficients in $\R^n$}. Ann. Scuola Norm. Sup. Pisa Cl. Sci. (4) \textbf{24}~(1997), 133-164.

\bibitem{Manfr} 
M. Manfredini,
\textit{The Dirichlet problem for a class of ultraparabolic equations}.
Adv. Differential Equations \textbf{2}~(1997), 831--866.

\bibitem{Mo1}
J. Moser, \emph{A Harnack inequality for parabolic differential equations}. Comm. Pure Appl. Math. \textbf{17}~(1964), 101-134.

\bibitem{Mo2}
J. Moser, \emph{Correction to: "A Harnack inequality for parabolic differential equations''}. Comm. Pure Appl. Math. \textbf{20}~(1967), 232-236.

\bibitem{nash}
J. Nash, \emph{Continuity of solutions of parabolic and elliptic equations}. Amer. J. Math. \textbf{80}~(1958), 931-954.

\bibitem{Ni}
S. M. Nikol'skii, \emph{Imbedding, continuation and approximation theorems for differentiable functions of several variables}. (Russian) Uspehi Mat. Nauk \textbf{16}~(1961), no. 5 (101), 63-114. 

\bibitem{Pcat94}
S. Polidoro, \emph{On a class of ultraparabolic operators of Kolmogorov-Fokker-Planck type}. Matematiche (Catania) \textbf{49}~(1994), no. 1, 53-105.

\bibitem{R}
M. Riesz, \emph{Int\'egrales de Riemann-Liouville et potentiels}. Acta Sci. Math. Szeged, \textbf{9}~(1938), 1-42.

\bibitem{Sc}
V. Scornazzani,
\textit{The Dirichlet problem for the Kolmogorov operator}. 
Boll. Un. Mat. Ital. C (5) \textbf{18}~(1981), 43-62.

\bibitem{Sickel}
W. Sickel, \emph{Pointwise multipliers of Lizorkin-Triebel spaces}. The Maz'ya anniversary collection, Vol. 2 (Rostock, 1998), 295-321, Oper. Theory Adv. Appl., 110, Birkh\"auser, Basel, 1999.

\bibitem{Str}
D.W. Stroock, \emph{An exercise in Malliavin's calculus}. J. Math. Soc. Japan \textbf{67}~(2015), no. 4, 1785-1799. 

\bibitem{T1}
M.H. Taibleson, \emph{On the theory of Lipschitz spaces of distributions on Euclidean $n$-space. I. Principal properties}. J. Math. Mech. \textbf{13}~(1964), 407-479.

\bibitem{T2}
M.H. Taibleson, \emph{On the theory of Lipschitz spaces of distributions on Euclidean $n$-space. II. Translation invariant operators, duality, and interpolation}. J. Math. Mech. \textbf{14}~(1965), 821-839. 

\bibitem{T}
E.C. Titchmarsh, \emph{The theory of functions}. Reprint of the second (1939) edition. Oxford University Press, Oxford, 1958. x+454 pp.

\bibitem{WZ} W. Wang \& L. Zhang,
\textit{The $C^\alpha$ regularity of weak solutions of ultraparabolic equations}.
Discrete Contin. Dyn. Syst. \textbf{29}~(2011), 1261-1275.

\bibitem{Web}
M. Weber, \emph{The fundamental solution of a degenerate partial differential equation of parabolic type}. Trans. Amer. Math. Soc \textbf{71}~(1951), 24-37. 
 
\end{thebibliography}

\end{document}